\documentclass[12pt]{amsart}
\usepackage{newtxtext,newtxmath,mathrsfs}
\usepackage{anysize}
\usepackage{amsmath}
\usepackage{amsthm}
\usepackage{stmaryrd}
\usepackage{mathrsfs}
\usepackage{dsfont}
\usepackage{soul}
\usepackage{color}
\usepackage{mathtools}
\usepackage{comment}
\usepackage{thmtools}
\usepackage{enumerate}
\usepackage{graphicx}
\usepackage{setspace}
\newcommand\mtop{1in}
\newcommand\mbottom{1in}
\newcommand\mleft{1in}
\newcommand\mright{1in}
\usepackage[top = \mtop, bottom = \mbottom, left = \mleft, right=\mright]{geometry}
\usepackage{subcaption}

\newtheorem{Theorem}{Theorem}[section]

\newtheorem{Lemma}[Theorem]{Lemma}

\newtheorem{Prop}[Theorem]{Proposition}

\theoremstyle{definition}
\newtheorem{Definition}{Definition}

\newcommand{\Z}{\mathbb{Z}}

\def\P{{\mathbb{P}}}
\newcommand{\E}{\mathbb{E}}

\usepackage{scalerel,stackengine}

\title{An estimate for the radial chemical distance in $2d$ critical percolation clusters}
\author{Philippe Sosoe}
\thanks{P.S.'s research is partially supported by NSF grant DMS 1811093.}
\address{Department of Mathematics, Cornell University. Malott Hall, Ithaca, NY 14853.}
\email{psosoe@math.cornell.edu}

\author{Lily Reeves}
\address{Center for Applied Mathematics, Cornell University. Frank H.T. Rhodes Hall, Ithaca 14850.}
\email{zw477@cornell.edu}

\begin{document}
\maketitle

\begin{abstract}
We derive an estimate for the distance, measured in lattice spacings, inside two-dimensional critical percolation clusters from the origin to the boundary of the box of side length $2n$, conditioned on the existence of an open connection. The estimate we obtain is the radial analogue of the one found in the work of Damron, Hanson, and Sosoe. In the present case, however, there is no lowest crossing in the box to compare to, so we construct a path $\gamma$ from the origin to distance $n$ that consists of ``three-arm'' points, and whose volume can thus be estimated by $O(n^2\pi_3(n))$. Here, $\pi_3(n)$ is the ``three-arm probability'' that the origin is connected to distance $n$ by three arms, two open and one dual-closed.  We then develop estimates for the existence of shortcuts around an edge $e$ in the box, conditional on $\{e\in \gamma\}$, to obtain a bound of the form $O(n^{2-\delta}\pi_3(n))$ for some $\delta>0$.
\end{abstract}
	
	\section{Introduction}
	
	\subsection{Chemical distance}
	We consider Bernoulli percolation on the two-dimensional lattice $\mathbb{Z}^2$ at the critical density $p_c=\frac{1}{2}$. In this context, the \emph{chemical distance} between two subsets $A$ and $B$ of $\mathbb{Z}^2$ is the least number of edges in any path of open edges connecting $A$ to $B$. We denote this distance by $\mathrm{dist}_c(A,B)$. 
	
	From the physics literature \cite{EK85,Grassberger99,HN84,HTWB85,HHS84,HS88,ZYDZ12}, it is expected that there exists an exponent $s>1$ such that if $A$ and $B$ are at Euclidean distance $n$, then:
	\begin{equation}\label{eqn: s-def}
		\mathbb{E}[\mathrm{dist}_c(A,B)\mid A\leftrightarrow B]\approx n^s.
	\end{equation}
	Here, $\{A\leftrightarrow B\}$ is the event that $A$ and $B$ are connected by a path of open edges. Despite remarkable progress on the derivation of other critical exponents, no approximation of the form \eqref{eqn: s-def} is  known for any reasonable interpretation of $\approx$. See Schramm's survey \cite{Schramm11} for a list of problems in percolation, where the determination of the exponent $s$ is listed as an important open problem. In particular, no clear connection has yet been discovered to the SLE process, which is used to derive other critical exponents in two-dimensional percolation. It is not at all clear how to relate the chemical distance, measured in lattice spacings, to any conformally invariant quantities.
	
	It is known from work of Aizenman and Burchard \cite{AB99} that, unlike in the supercritical case, the chemical distance for critical percolation is super-linear: there is $\eta>0$ such that, with high probability
	\[\mathrm{dist}_c(A,B)\ge n^{1+\eta}\]
	if $A$ and $B$ are at Euclidean distance greater than or equal to $n$. The size of $\eta$ is not made explicit in \cite{AB99} and no other lower bound is known.
	
	In \cite{KZ93}, Kesten and Zhang noted that if one restricts attention to paths inside a square box $B_n=[-n,n]\times[-n,n]$, and lets $A=\{-n\}\times [-n,n]$ and $B=\{n\}\times [-n,n]$ be the two vertical sides of $B_n$, one obtains an upper bound by considering the \emph{lowest crossing} $\ell_n$ of $B_n$:
	\begin{equation*}
		\mathbb{E}[\mathrm{dist}_c(A,B)\mid A\leftrightarrow B]\le \mathbb{E}[\# \ell_n \mid A\leftrightarrow B].
	\end{equation*}
	By now standard computations, the expected size of the lowest crossing can be expressed in terms of the \emph{three-arm probability} to distance $n$:
	\begin{equation}\label{eqn: tri-bound}
		\mathbb{E}[\#\ell_n \mid A\leftrightarrow B] \le Cn^2\pi_3(n),
	\end{equation}
	see for example Morrow, Zhang \cite{MZ05}. Here $\pi_3(n)=\mathbb{P}(A_3(n))$ is the probability that there are two open and one closed dual connections from the origin to the boundary of the box $B_n$. On the triangular lattice, the corresponding probability is computed in \cite{SW01}:
	\[\pi_3(n)= n^{-\frac{2}{3}+o(1)}.\]
	
	In \cite{DHS17,DHS21}, Damron, Hanson, and the first author answered a question of Kesten-Zhang, showing that the upper bound \eqref{eqn: tri-bound} can be improved by a factor $n^{-\delta}$, for some $\delta>0$:
	\begin{equation}\label{eqn: DHS-bound}
		\mathbb{E}[\mathrm{dist}_c(A,B)\mid A\leftrightarrow B]\le Cn^{-\delta}\mathbb{E}[\# \ell_n \mid A\leftrightarrow B].
	\end{equation}
	
	Rather than the shortest crossing across a box, in this paper we consider the expected distance from the origin to the boundary of a box $B_n$, conditioned on the existence of an open connection. Such a ``radial'' estimate is more natural than the chemical distance across a box in many applications, notably the study of random walks on percolation clusters. Unlike in the case of a horizontal crossing, there is no natural crossing to compare to in this case. Nevertheless, we show that a bound of the form \eqref{eqn: DHS-bound} also holds for the radial chemical distance.
	\begin{Theorem} \label{thm:main}
		Let $\{0\leftrightarrow \partial B_n\}$ be the event that there is an open connection from the origin $(0, 0)$ to $\partial B_n$. On $\{0\leftrightarrow \partial B_n\}$, the random variable $S_n$ is defined as the chemical distance between the origin and $\partial B_n$.
		
		There exist some $\delta >0$ and constant $C > 0$ independent of $n$ such that
		\begin{equation} \label{eqn:mainresult}
			\E[S_n\mid 0\leftrightarrow \partial B_n] \leq Cn^{2-\delta}\pi_3(n).
		\end{equation}
    \end{Theorem}
    Note that, in addition to the absence of a lowest crossing to compare to, the conditioning in the estimate \eqref{eqn:mainresult} is singular, in that the probability $\mathbb{P}(0\leftrightarrow \partial B_n)$ tends to zero with $n\rightarrow \infty$. In the next subsection, we explain further why the arguments of the present paper differ essentially from \cite{DHS17,DHS21} despite the similarity in the results (c.f. in particular Sections \ref{sec: 3-arm} and \ref{sec:improvements}).

d    Simulation results in \cite{ZYDZ12} suggest that the critical exponent $s$ in \eqref{eqn: s-def} for the chemical distance is approximately $1.1308$. The value of $\pi_3(n)$ on the square lattice is expected to be on the order of $n^{-2/3}$, as on the triangular lattice, although this has not yet been proved. Thus, we expect the optimal value of $\delta$ to be approximately $0.2$. This is out of reach of current methods.
	
	\subsection{Overview of the paper}
	
	The key estimate from \cite{DHS21}, \eqref{eqn: strict-main} in Section \ref{section:proofofthm}, provides a method to build a path of expected length of $n^{2-\delta}\pi_3(n)$ by finding, with high probability, shortcuts around paths consisting of ``three-arm'' edges. Thus one gains the $n^{-\delta}$ factor in \eqref{eqn:mainresult}. Using this estimate requires two inputs, one of which, the existence of an open connection of volume $O(n^2\pi_3(n))$, is a given in the case of crossing paths in \cite{DHS17, DHS16, DHS21} but is non-trivial in our case. The second, an estimate for the probability of shortcuts around an edge $e$, conditional on the edge belonging to a three-arm path, leads to considerably more involved constructions than in the aforementioned works.
	
	Concerning the first point above, in the radial case, unlike the case considered in \cite{DHS21}, there is no canonical path of three-arm points connecting the origin to the boundary of $B_n$. In Section \ref{sec: 3-arm}, we present and complete a construction from the unpublished note \cite{DHS16} to find a path $\gamma$ of expected length $n^2\pi_3(n)$ conditional on the existence of a path from the origin to $\partial B_n$. The idea is to consider successive innermost circuits around the origin, and build three-arm paths connecting these. 
	
	
	Secondly, we need an estimate for the conditional probability that there is no shortcut around an edge $e\in B_n$, conditioned on $e\in \gamma$. We show that this probability is bounded up to a constant factor by the probability of the same shortcut event conditioned on a three-arm event centered at $e$. The requisite gluing constructions appear in Section \ref{sec:improvements}. Given that the construction of $\gamma$ is rather more complicated than that of the lowest crossing, the RSW/generalized FKG estimates here are more involved than the ones used to obtain the corresponding comparison in \cite{DHS17, DHS16, DHS21}. In particular, forcing the occurrence of the event $e\in \gamma$ is a more delicate matter than in these works.

	\subsection{Notations} \label{subsection:notations}
	In this section, we summarize the notations we will use. We mostly follow the conventions established in \cite{DHS16} and \cite{DHS21}.
	
	Throughout the paper, we consider Bernoulli percolation on the square lattice $\Z^2$ seen as a graph with the edge set $\mathcal{E}$ consisting of all pairs of nearest-neighbor vertices. 
	
	We let $\P$ be the critical bond percolation measure 
	\[\P = \prod_{e\in \mathcal{E}} \frac12 (\delta_0+\delta_1)\] 
	on the state space $\Omega = \{0,1\}^{\mathcal{E}}$, with the product $\sigma$-algebra. An edge $e$ is said to be \emph{open} in the configuration $\omega\in \Omega$ if $\omega(e) = 1$ and \emph{closed} otherwise. 
	
	A (lattice) \textit{path} is a sequence $(v_0, e_1, v_1, \dots , v_{N-1}, e_N , v_N)$ such that for all $k = 1,\dots,N$, $\|v_{k-1} - v_k\|_1 = 1$ and $e_k = \{v_{k-1},v_k\}$. 
	A \textit{circuit} is a path with $v_0 = v_N$. 
	Given $\omega \in \Omega$, we say that $\gamma=(e_k)_{k=1,\dots,N}$ is open in $\omega$ if $\omega(e_k) = 1$ for $k = 1,\dots ,N$.
	
	
	
	The dual lattice is written $((\Z^2)^*, \mathcal{E}^*)$, where 
	\[(\Z^2)^* = \Z^2 + \left( \frac12,\frac12 \right)\] 
	and $\mathcal{E}^*$ its nearest-neighbor edges. 
	
	Given $\omega \in \Omega$, we define $\omega^* \in \Omega^* = \{0,1\}^{\mathcal{E}^*}$ by the relation $\omega^*(e^*) = \omega(e)$, where $e^*$ is the dual edge that shares a midpoint with $e$. The edge $e^*$ is said to be \emph{dual-closed} if $\omega(e)=0$.
	For any $V \subset \mathbb{R}^2$ we write 
	\[V^*=V+\left(\frac{1}{2},\frac{1}{2}\right).\]
	
	For $x \in \Z^2$, we define
	\[B(x,n)=\{(x_1, x_2)\in \mathcal{E} :x_1\sim x_2, |x_1-x|_\infty \le n, |x_2-x|_\infty \le n\}.\]
	$x\sim y$ means $x$ and $y$ are nearest neighbors on the lattice $\Z^2$.
	When $x$ is the origin $(0,0)$, we sometimes abbreviate $B((0,0),n)$ by $B_n$ or $B(n)$. We denote by $\partial B(x,n)$ the set
	\[\partial B(x,n)=\{(x_1, x_2)\in \mathcal{E} :x_1\sim x_2, |x_1-x|_\infty = n, |x_2-x|_\infty = n\}.\]
	
	In this paper, we sometimes abuse notation and write $B(e,n)$ for an edge $e$ to mean the box $B(e_x,n)$ where $e_x$ is the lower-left endpoint of $e$, defined as the first of the two endpoints of $e$ in the lexicographic order on $\mathbb{Z}^2$.
	
	For the purpose of this paper, we define an annulus centered at $x \in \Z^2$ as the difference between two boxes of different sizes centered at $x$:
	\[\text{For $0<n<N$, }  B(x,n,N) = B(x,N) \setminus B(x,n).\]
	We often abbreviate $B((0,0),n,N)$ as $B(n,N)$ when $x=(0,0)$ is implied.
	
	Distances in this paper are measured in the $\ell_\infty$ norm 
	\begin{equation*}
	\|x-y\|_\infty = \max_{i=1,2}|x_i-y_i|, \quad \text{where $x,y\in \Z^2$}.
	\end{equation*}
	
	
	A \emph{color sequence} $\sigma$ of length $k$ is a sequence $(\sigma_1, \dots, \sigma_k)\in \{O, C\}^k$. Each $\sigma_i$ indicates a ``color'', with $O$ representing open and $C$ representing dual closed. 
	
	
	For $n\leq N$, we define a $k$-arm event with color sequence $\sigma$ to be the event that there are $k$ disjoint paths whose colors are specified by $\sigma$ in the annulus $B(n,N)$ connecting $\partial B_n$ and $\partial B_N$. Formally,
	\begin{equation*}
		A_{k,\sigma}(n,N) := \{\partial B_n \leftrightarrow_{k,\sigma} \partial B_N\}.
	\end{equation*}
	
	We note a technical point: for $A_{k,\sigma}(n,N)$ to be defined, $n$ needs to be big enough for all $k$ arms to be (vertex)-disjoint. We define $n_0(k)$ to be the smallest integer such that $|\partial B(n_0(k))| \geq k$. Color sequences that are equivalent up to cyclic order denote the same arm event.

    In Section \ref{sec:improvements}, we use \emph{half-plane} versions of the arm events above. The half-plane event $A_{k,\sigma}^{hp}(n,N)$ is the event that $A_{k,\sigma}(n,N)$ occurs and all arms are contained in a half-plane $\{(x,y)\in \mathbb{R}^2: v\cdot(x,y)\le 0\}$ for some unit vector $v\in \{(0,1),(1,0)\}$.
       
    We use special notation for the probabilities of certain arm events that occur frequently. We denote by $\pi_2(n,N)$ the two-arm probability for two arms, one open and one closed dual:
    \[\pi_2(n,N):=\P(A_{2,OC}(n,N));\]
    we also let $\pi_3(n,N)$ be the three-arm probability for the event that there are two open arms and one closed dual arm in $B(n,N)$
    \[\pi_3(n,N):=\P(A_{3,OOC}(n,N)).\]
    For $k\ge 3$, we denote by $\pi_k(n,N)$ the event that there are $k-1$ disjoint open and one closed dual arm joining the boundaries of the annulus.
    Monochromatic $k$-arm probabilities are denoted by $\pi'_k$:
    \[\pi'_k(n,N):=\P(A_{k,O\cdots O}(n,N))=\P(A_{k,C\cdots C}(n,N)).\]

A crucial feature of arm events is \emph{quasi-multiplicativity}, expressed in the following Proposition. See \cite[Proposition 12]{Nolin08}.
\begin{Prop}\label{prop: smoothness}
  Let $k\ge 1$ be an integer, and $\sigma\in \{O,C\}^k$ a color sequence. Then there are constants $c,C>0$ such that, uniformly in $n$,
  \[c\P(A_{k,\sigma}(n/2,N))\le \P(A_{k,\sigma}(n,N))\le C\P(A_{k,\sigma}(n,2N)),\]
  and, for $n<n'<N$,
    \begin{equation}\label{eqn: smoothness}
    \P(A_{k,\sigma}(n,N))\ge c\P(A_{k,\sigma}(n,n'))\P(A_{k,\sigma}(n',N)).
    \end{equation}
\end{Prop}
	
	Generally, we reserve the letter $A$ for arm events, the letter $B$ for boxes, and the capital letter $C$ (with various fonts) for circuits and events related to circuits, for ``closed'' when used as a subscript, and at times for various constants, along with the small letter $c$. All other notations will be specified as needed.
	
	\subsection{The standard gluing construction} \label{subsection:gluing}
	
	One technique that we will repeatedly use in the proof, presented in Section \ref{sec:improvements}, of the main estimate Proposition \ref{prop:main}, is the standard gluing construction using the generalized Fortuin-Kasteleyn-Ginibre (FKG) inequality combined with Russo-Seymour-Welsh (RSW) estimates. This methodology was first applied extensively in H. Kesten's papers on critical percolation \cite{Kesten86}, \cite{Kesten87}. Let us begin by stating the FKG and RSW estimates.
	
	\begin{Theorem}[Generalized FKG, \cite{Nolin08}]
		Consider two increasing events $A^+$, $\tilde{A}^+$, and two decreasing events $A^-$, $\tilde{A}^-$. Assume that there exist three disjoint finite sets of vertices $\mathcal{A}$ , $\mathcal{A}^+$, and $\mathcal{A}^-$ such that $A^+$, $A^-$, $\tilde{A}^+$, and $\tilde{A}^-$ depend only on the sites in, respectively, $\mathcal{A} \cup\mathcal{A}^+$, $\mathcal{A}\cup \mathcal{A}^-$, $\mathcal{A}^+$, and $\mathcal{A}^-$. Then we have
		\begin{equation}\label{eqn: FKG}
			\hat{\P}(\tilde{A}^+ \cap \tilde{A}^-\mid A^+ \cap A^-) \geq \hat{\P}(\tilde{A}^+)\hat{\P}(\tilde{A}^-) 
		\end{equation}
		for any product measure $\hat{\P}$ on $\Omega$.
    \end{Theorem}

    \begin{Theorem}[Russo-Seymour-Welsh, \cite{Nolin08}]
        Let $k>0$, and let $H_k(n)$ be the event that there is a horizontal open crossing of the rectangle $[0,kn]\times[0,n]$. There exists $\delta_k>0$ such that
        \begin{equation}\label{eqn: RSW}
            \mathbb{P}(H_k(n))\ge \delta_k, \quad n\ge 1.
        \end{equation}
    \end{Theorem}
    
    We illustrate the combination of the above results with an elementary construction showing that two pairs of arms, two open and two closed dual, from the origin to distance $n$ and across the annulus $B(n,N)$, respectively, can be glued to form the event $A_{2,OC}(N)$. The example is represented schematically in Figure \ref{fig: drawing13}. The blue box denotes the set of vertices $\mathcal{A}^+$, the red box denotes the set of vertices $\mathcal{A}^-$, and the rest of $B_N$ is $\mathcal{A}$. Then, $A^+$ is the event that there are the two open arms, $A^-$ the two closed arms, with $\tilde{A}^+$ and $\tilde{A}^-$ signifying the crossings in $\mathcal{A}^+$ and $\mathcal{A}^-$ respectively. Then, generalized FKG states we have
	\begin{equation*}
		\P(\tilde{A}^+ \cap \tilde{A}^-\cap A^+ \cap A^-) \geq \P(\tilde{A}^+)\P(\tilde{A}^-) \P( A^+ \cap A^-).
	\end{equation*}
	By \eqref{eqn: RSW}, one has $\P(\tilde{A}^+)\geq \delta >0$ and $\P(\tilde{A}^-)\geq \delta' >0$. In this example, we need additional arm separation conditions developed in \cite{Kesten87} to ensure that the arms connect at the boundary. It is well-known that the separated arm probabilities are on the same order as regular arm probabilities, see \cite[Theorem 11]{Nolin08}. For these versions of the events with an extra separation condition, we have:
	\begin{equation*}
	\begin{split}
	    \P(A_{2,OC}(N)) 
	    &\geq c\P(A_{2,OC}^{\mathrm{sep}}(n),A_{2,OC}^{\mathrm{sep}}((1+\kappa)n,N)) \\
	    &\ge c\P(A_{2,OC}^{\mathrm{sep}}(n)) \P(A_{2,OC}^{\mathrm{sep}}((1+\kappa)n,N))\\ 
	    &\ge c\P(A_{2,OC}^{\mathrm{sep}}(n))\P(A_{2,OC}^{\mathrm{sep}}(n,N)),
	\end{split}
    \end{equation*}
    where the generalized FKG inequality is applied to both the first and second lines and Proposition \ref{prop: smoothness} is applied to the third line. We note that $c$ is uniform in $n$. Therefore,
    \begin{equation*}
        \P(A_{2,OC}(N)) \ge c \P(A_{2,OC}(n))\P(A_{2,OC}(n,N)).
    \end{equation*}

\begin{figure}[hpt]
	\centering
	\includegraphics[width=0.45\textwidth]{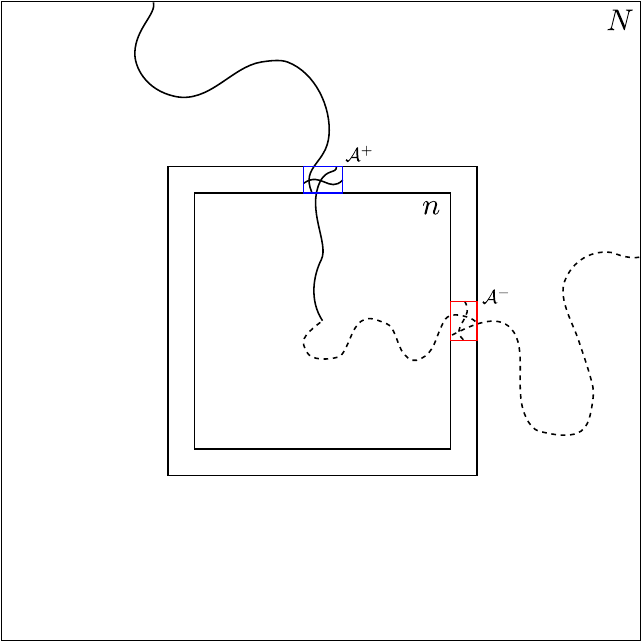}
	\caption{An example of a gluing construction using the generalized FKG and RSW estimates.
	}
	\label{fig: drawing13}
\end{figure}

\subsection{Acknowledgements} We wish to thank Michael Damron and Jack Hanson, with whom some of the ideas appearing here were developed and recorded in the unpublished note \cite{DHS16}. We also thank the anonymous reviewer for their careful reading of the manuscript.
        
\section{A three-arm path to $\partial B_n$}\label{sec: 3-arm}
	The first step towards proving our main result is to find a replacement for the lowest path in \cite{DHS17}. Here we use ideas in the (unpublished) note \cite{DHS16}. Note however that, although the argument in \cite{DHS16} is presented on the square lattice $\Z^2$, it relies on an inequality of Beffara and Nolin \cite{BN11} comparing monochromatic and polychromatic arm events. This is expected to hold on $\Z^2$, but is so far known only on the triangular lattice. The modified argument we give here holds unconditionally for the square lattice.
	
	\subsection{Definition of $\gamma$}\label{sec: gamma-def} We start by defining the central object $\gamma$, a path from the origin to $\partial B_n$. On the event $\{0\leftrightarrow \partial B_n\}$, let $C_0$ be the event that there is an open circuit around the origin in $B_n$. The definition of $\gamma$ and our estimate for $S_n$ will depend on whether $C_0$ occurs. 
	
	 Let us first fix a deterministic ordering on all edges in $B_n$. On $\{0\leftrightarrow \partial B_n\} \cap C_0^c$, there exists a closed dual arm from the origin to the boundary. Then, let $\mathfrak{c}$ be the first such path in the lexicographical order of paths viewed as sequences of edges. We let $\gamma$ be the open arm from the origin to $\partial B_n$ closest to the counterclockwise side of $\mathfrak{c}$. Here, ``closest'' is measured by the number of edges (or area) between $\gamma$ and $\mathfrak{c}$. 
	 
	On the event $C_0$, we construct an open path $\gamma$ from $0$ to $\partial B_n$ based on an idea in \cite[Section 2.3]{DHS16}. We first need a number of definitions. 
		
	We begin by enumerating the successive innermost open circuits around the origin. We denote the number of such circuits by $\mathcal{K}$. For $1\le m\le \mathcal{K}$, we denote by $\mathcal{C}_m$ the $m$-th innermost open circuit.
		
	By definition of $\mathcal{C}_1$, there is a closed dual path $\mathfrak{c}_1$ from a dual neighbor of the origin to the endpoint of the dual edge of an edge of the inner-most circuit $\mathcal{C}_1$. On $\{0\leftrightarrow \partial B_n\}$, there is also an open path from the origin to $\mathcal{C}_1$. We let $\sigma_1$ be the open path from the origin to $\mathcal{C}_1$ that is closest to the counterclockwise side of $\mathfrak{c}_1$. Similarly, for $m=2,\ldots,\mathcal{K}$, there is a closed dual path inside the region bounded by $\mathcal{C}_m$ but outside $\mathcal{C}_{m-1}$ joining the endpoint of a dual edge to $\mathcal{C}_{m-1}$ to the endpoint of a dual edge to $\mathcal{C}_m$. We let $\mathfrak{c}_m$ be the first such dual path, and $\sigma_m$ be the open path connecting $\mathcal{C}_{m-1}$ to $\mathcal{C}_m$ that is closest to the counterclockwise side of $\mathfrak{c}_m$. Finally, by the definition of $\mathcal{K}$, there necessarily exists a dual path joining a dual edge to $\mathcal{C}_{\mathcal{K}}$ to $\partial B_n$. We let $\mathfrak{c}_{\mathcal{K}+1}$ be the first such dual path. We define $\sigma_{\mathcal{K}+1}$ to be the open path joining $\mathcal{C}_{\mathcal{K}}$ to $\partial B_n$ that is closest to the counterclockwise side of $\mathfrak{c}_{\mathcal{K}+1}$.
		
	The definition of the path $\gamma$ from $0$ to $\partial B_n$ on the event $C_0$ is now as follows. The initial segment of $\gamma$ is $\sigma_1$. From the endpoint of $\sigma_1$ on $\mathcal{C}_1$, we follow the circuit $\mathcal{C}_1$ in the counterclockwise direction to the endpoint of $\sigma_2$ on $\mathcal{C}_1$. Then, we repeat this procedure for $m=2, \ldots,\mathcal{K}$, concatenating $\sigma_m$ with the open subpath $\tilde{\mathcal{C}}_m$ of $\mathcal{C}_m$ obtained by following this circuit in the counterclockwise direction, joining the endpoint of $\sigma_m$ belonging to $\mathcal{C}_m$ to the endpoint $\sigma_{m+1}$ on $\mathcal{C}_{m}$. Finally, we add $\sigma_{\mathcal{K}+1}$ to the path $\gamma$. 
	The main result of this section is the following:
	\begin{Lemma} \label{lemma:DHS16}
		There is a constant $C > 0$ independent of $n$ such that
		\begin{equation*}
		\E[\# \gamma \mid 0\leftrightarrow \partial B_n] \leq Cn^2\pi_3(n).
		\end{equation*}
	\end{Lemma}
	Clearly, we then have
	\begin{equation}\label{eqn: simple-comp}
	S_n\le \# \gamma,
	\end{equation}
    from which it follows that the distance $S_n$ is $O(n^2\pi_3(n))$ in expectation.
	In our construction, we frequently refer to the following quantity $M=M(e)$, defined for each edge $e$ inside $B_n$ by
	\begin{equation} \label{eqn: M-def}
	M(e) := \min(\mathrm{dist}(e,0), \mathrm{dist}(e, \partial B_n)).
	\end{equation}
    The proof of  Lemma \ref{lemma:DHS16} occupies the remainder of Section \ref{sec: 3-arm}.
		
		\subsection{Estimate on the event $C_0^c$.}\label{sec: C0c} 
		
		By duality, for each edge $e\in \gamma$, there is a closed dual path from an endpoint of $e^*$ to  $\mathfrak{c}$. Thus, for each edge $e\in \gamma$, there are two open arms and a closed dual arm from $e$ to distance $M(e)$. The open arms are obtained by following $\gamma$ from $e$ in either direction, and the closed dual arm is obtained by following the closed dual path from $e^*$ to $\mathfrak{c}$, and then following $\mathfrak{c}$ to the origin or $\partial B_n$. Combining this fact with \eqref{eqn: simple-comp}, we obtain:
		
		\begin{equation}\label{eqn: sum-over-M}
		\begin{split}
		\E[\#\gamma , C_0^c \mid 0\leftrightarrow \partial B_n]
		\leq C \sum_{k=1}^{\lfloor n/2\rfloor} \sum_{e:\ M(e)=k} \P(A_3(e,k)\mid 0\leftrightarrow \partial B_n).
		\end{split}
		\end{equation}
		
		Let $d=\mathrm{dist}(0,e)$. If in the inner sum $d < 2k$, $e\in B(2n/3)$. By independence, we have
		\begin{equation}\label{eqn: quasi-mult}
		\P(A_3(e,k) \mid 0\leftrightarrow \partial B_n)\le \frac{\P(0\leftrightarrow \partial B(d/2)) \P(A_3(e, k/2))\P(\partial B(3d/2) \leftrightarrow \partial B_n)}{\mathbb{P}(0\leftrightarrow \partial B_n)}.
		\end{equation}
		By quasi-multiplicativity, we have for some constant $C>0$,
		\begin{equation}\label{eqn: quasi-mult-2}
		\P(0\leftrightarrow \partial B_n)\geq C\P(0\leftrightarrow \partial B(d/2)) \P(\partial B(3d/2) \leftrightarrow \partial B_n).
		\end{equation}
		
		If $d\ge 2k$, then again by independence, we have
		\begin{equation} \label{eqn:indep-case-b}
		    \P(A_3(e,k) \mid 0\leftrightarrow \partial B_n) \le \frac{\P(0\leftrightarrow \partial B(n/2)) \P(A_3(e,k/2))}{\P(0\leftrightarrow \partial B_n)}.
		\end{equation}
		Similarly, by Proposition \ref{prop: smoothness}, we have for some constant $C>0$,
		\begin{equation} \label{eqn:quasi-mult-3}
		    \P(0\leftrightarrow \partial B_n) \ge C\P(0\leftrightarrow \partial B(n/2)).
		\end{equation}
		
		Using \eqref{eqn: quasi-mult}, \eqref{eqn: quasi-mult-2}, \eqref{eqn:indep-case-b}, and \eqref{eqn:quasi-mult-3} to estimate \eqref{eqn: sum-over-M}, we have 
		\begin{equation}\label{eqn: before-smooth}
		\begin{split}
		\E[\# \gamma, C_0^c \mid 0\leftrightarrow \partial B_n] &\leq C\sum_{k=1}^{\lfloor n/2\rfloor}\sum_{e:M(e)=k}\mathbb{P}(A_3(e,k/2)) \\
		&\leq C \sum_{k=1}^{\lfloor n/2\rfloor} k\pi_3(k).
		\end{split}
		\end{equation}
		To bound the final quantity, we use the following result, proved for example in \cite[Eqn. (7)]{Kesten86} and \cite[Proposition 16]{DHS17}.
		
		\begin{Prop} \label{prop:smoothness}
		There exists $C>0$ such that for all $L \in \Z^+$,
			\begin{equation*}
			\sum_{l=1}^L l\pi_3(l) \leq CL^2\pi_3(L).
			\end{equation*}
		\end{Prop}
		Applying this to \eqref{eqn: before-smooth}, we immediately obtain the desired bound
		\[\E[\# \gamma, C_0^c \mid 0\leftrightarrow \partial B_n]  \leq Cn^2\pi_3(n).\]	
		
		\subsection{Estimate on the event $C_0$.}\label{sec: C0}
		We estimate the expected volume $\mathbb{E}[\#\gamma\mid 0\leftrightarrow \partial B_n]$. Proceeding as in the previous case $C_0^c$ (see \eqref{eqn: sum-over-M}), it suffices to estimate the conditional probability that an edge $e\in B_n$ belongs to any portion of $\gamma$.
	
		\subsubsection{The initial path $\sigma_1$} 
		If $e\in \sigma_1$, then by duality there is a closed dual path from a dual neighbor of $e$ to the closed dual path $\mathfrak{c}_1$, which then follows $\mathfrak{c}_1$ to a dual neighbor of the origin. There are also two open arms from $e$ to the origin and to $\mathcal{C}_1$ respectively. All three arms reach to distance at least $d=\mathrm{dist}(e,0)$. Recalling the notation \eqref{eqn: M-def}, we have the estimate:
		\begin{equation}\label{eqn: rafaela}
		\P(e \in \sigma_1 \mid 0\leftrightarrow \partial B_n) \le  \P(A_3(e, d)) \le  C \pi_3(M(e)).
		\end{equation}
		
		\subsubsection{The final path $\sigma_{\mathcal{K}+1}$}
		
		For an edge $e \in \sigma_{\mathcal{K}+1}$, by duality there is a closed dual path from a  dual neighbor of $e$ to the closed dual path $\mathfrak{c}_{\mathcal{K}+1}$, which then follows $\mathfrak{c}_{\mathcal{K}+1}$ to $\partial B_n$. There are also two open arms from $e$ to $\partial B_n$ and $\mathcal{C}_\mathcal{K}$ respectively. As in the previous case, from $e$, there are three arms to distance at least $M(e)$ and we obtain the bound 
		\begin{equation}\label{eqn: gabriela}
		\P(e \in \sigma_{\mathcal{K}+1} \mid 0\leftrightarrow \partial B_n) \leq C \pi_3(M(e)).
		\end{equation}
		
		\subsubsection{The intermediate paths $\sigma_m$ for $m=2,\dots,\mathcal{K}$} \label{sec: intermediate}
		In this case, we use a modified version of Lemma $11$ from \cite{DHS16}.
		\begin{Lemma}\label{lemma:A_3}
			Suppose $e \in \sigma_m$ for $2 \le m\leq \mathcal{K}$. Fix $\epsilon>0$ and an integer $R$ such that for any $0\le n_1<n_2$,
			\begin{equation} \label{eqn:stopcondition}
				\pi_{2R+2}'(n_1,n_2)\leq \pi_3(n_1,n_2) (n_1/n_2)^{\epsilon}.
			\end{equation} 
			There are $0=\ell_0 < \ell_1\leq \dots \leq \ell_{R} \leq \lfloor \log M(e) \rfloor$ such that
			\begin{enumerate}
				\item There are two disjoint open arms and one closed dual arm from $e$ to $\partial B(e,2^{\ell_1-1})$.
				\item For $i\geq 2$, if $\ell_{i-1} < \lfloor\log M(e)\rfloor$, there are $2i$ open arms and one closed dual arm from $\partial B(e, 2^{\ell_{i-1}})$ to $\partial B(e, 2^{\ell_i-1})$, all of which are disjoint.
				\item If $\ell_{R} < \lfloor \log M(e)\rfloor$, there are $2R+2$ disjoint open arms from $\partial B(e, 2^{\ell_{R}})$ to $\partial B(e, M(e))$.
			\end{enumerate}
		\end{Lemma}
		
		\begin{proof}
			We first note that such an $R$ that satisfies \eqref{eqn:stopcondition} exists by the van den Berg-Kesten inequality and because $\pi_3(n_1, n_2)\geq (n_1/n_2)^{c}$ for some constant $c$. See \cite[Proposition 14]{Nolin08}. 
			
			We observe that $e \in \sigma_m$ lies between the two circuits $\mathcal{C}_{m-1}$ and $\mathcal{C}_m$. Let $x$ be the endpoint of $\sigma_m$ intersecting $\mathcal{C}_{m-1}$, and let $y$ be the other endpoint intersecting $\mathcal{C}_m$. Let $\ell''$ be the smallest $\ell'$ such that $2^{\ell'} \geq M$ or $B(e,2^{\ell'})$ contains both $x$ and $y$. 
			
			Suppose first that $2^{\ell''} \geq M$. In this case, we let $\ell_1 = \lfloor \log M\rfloor$. Subsequently, we let all $\ell_i = \lfloor \log M\rfloor$ for $i>1$. The two endpoints of $\sigma_m$ inside $B(e, 2^{\ell_1})$ form two open arms from $e$ to $\partial B(e,2^{\ell_1})$. The edge of $\sigma_m$ with endpoint $y$ has a dual edge connected by a dual closed path to $\mathfrak{c}_m$. Since the same is true of the edge $e^*$, we obtain a closed arm from $e^*$ which extends at least to distance $\mathrm{dist}(e, y) \geq 2^{\ell_1-1}$.
			 
			If $2^{\ell''} < M$, we let $\ell_1 = \ell''$. As in the previous case, we obtain two open arms and a dual closed arm from $e$ to distance $2^{\ell_1-1}$. In addition, both $x$ and $y$ are contained in $B(e,2^{\ell_1})$, and each lies on one of the open circuits $\mathcal{C}_{m-1}$ and $\mathcal{C}_{m}$ around $0$, whereas $0 \not\in B(e,2^{\ell_1})$ since $2^{\ell_1} < M$. 
			
			If $m\leq 2$, we let $\ell_2 = \lfloor \log M\rfloor$. Otherwise, if $2^{\ell_1}<M$, we let $\ell''$ be the least $\ell'\ge \ell_1$ such that $B(e, 2^{\ell'})$ intersects $\mathcal{C}_{m-2}$. If $2^{\ell''} \geq M$, we let $\ell_2 = \lfloor \log M\rfloor$, and if $2^{\ell''} < M$, we let $\ell_2 = l''$. In all three cases we have four open arms (following $\mathcal{C}_{m-1}$ and $\mathcal{C}_m$) and one closed dual arm (the endpoint $x$ has a dual edge connected by a dual closed path to a dual edge of $\mathcal{C}_{m-2}$) from $\partial B(e, 2^{\ell_1})$ to $\partial B(e, 2^{\ell_2-1})$.
			
			Inductively, unless $\ell_{i-1} = \lfloor \log M\rfloor$, we define $\ell_i$ to be the least such that $B(e, 2^{\ell_i})$ intersects $\mathcal{C}_{m-i}$. For each scale $\ell_i$, if $\ell_{i} < \lfloor \log M\rfloor$, the box $B(e, 2^{\ell_i})$ crosses one more circuit than $B(e, 2^{\ell_{i-1}})$, thus resulting in two more open arms. At scale $\ell_{R}$, one may not have a closed dual arm from $\partial B(e, 2^{\ell_{R}})$ to $\partial B(e, M(e))$, but there are $2R+2$ disjoint open arms since the box $B(e, 2^{\ell_{R}})$ crosses $R+1$ circuits, each resulting in two open arms. If for some $i$, $\ell_{i-1} = \lfloor \log M\rfloor$, we define all subsequent scales to be $\lfloor \log M\rfloor$. 
		\end{proof}
	On the event $\left\{e\in \big(\cup_{m=2}^\mathcal{K} \sigma_m\big)\right\}$, by Lemma \ref{lemma:A_3}, computations similar to \eqref{eqn: quasi-mult}, \eqref{eqn: quasi-mult-2},  and a union bound over $0=\ell_0 < \ell_1\leq \dots \leq \ell_{R} \leq \lfloor \log M(e) \rfloor$, we have
		\begin{equation}\label{eqn:adelaide}
			\P\big(e\in \cup_{m=2}^\mathcal{K} \sigma_m \mid 0\leftrightarrow \partial B_n \big) 
			\leq C\sum_{0=\ell_0 < \ell_1\leq \dots \leq \ell_{R}}^{\lfloor \log M(e) \rfloor} \prod_{i=1}^R \pi_{2i+1}(2^{\ell_{i-1}}, 2^{\ell_i-1}) \pi_{2R+2}'(2^{\ell_{R}},M).
		\end{equation}
		
		Using Reimer's inequality in the form
		\[\pi_{2i+1}(2^{\ell_{i-1}},2^{\ell_i-1})\le \pi_3(2^{\ell_{i-1}}, 2^{\ell_i-1})\pi_{2i-2}'(2^{\ell_{i-1}}, 2^{\ell_i-1}),\]
		Proposition \ref{prop: smoothness}, and \eqref{eqn:stopcondition} we have
		\begin{align} \label{eqn:adelaide-2}
			\sum_{0=\ell_0 < \ell_1\leq \dots \leq l_{R}}^{\lfloor \log M(e) \rfloor} &\prod_{i=1}^{R} \pi_{2i+1}(2^{\ell_{i-1}}, 2^{\ell_i-1})\pi_{2R+2}'(2^{\ell_{R}},M) \nonumber\\
			&\leq C'\sum_{0=\ell_0 < \ell_1\leq \dots \leq l_{R}}^{\lfloor \log M(e) \rfloor} \prod_{i=1}^R \pi_3(2^{\ell_{i-1}}, 2^{\ell_i}) \pi_3(2^{\ell_{R}}, M) \prod_{i=1}^{R} \pi_{2i-2}'(2^{\ell_{i-1}}, 2^{\ell_i}) \left(\frac{2^{\ell_{R}}}{M}\right)^\epsilon.
		\end{align}
		We are then able to use the standard gluing construction $R$ times between each pair of consecutive scales for the polychromatic three-arm event:
		\begin{equation*}
			\prod_{i=1}^R \pi_3(2^{\ell_{i-1}}, 2^{\ell_i}) \pi_3(2^{\ell_{R}}, M) \leq C^R \pi_3(M(e)).
		\end{equation*}
		For the monochromatic probabilities in \eqref{eqn:adelaide-2}, we have the following estimate:
		\begin{equation*}
			\pi_j'(n_1,n_2)\leq (\pi_1(n_1,n_2))^j \leq \left(\frac{n_1}{n_2}\right)^{\alpha j},\, \text{for some $\alpha >0$}.
		\end{equation*}
		Plugging this back into \eqref{eqn:adelaide-2}, we have
		\begin{align*}
			\eqref{eqn:adelaide-2}
			&\leq C'C^R\pi_3(M(e)) \sum_{0=\ell_0 < \ell_1\leq \dots \leq \ell_{R}}^{\lfloor \log M(e)\rfloor} \prod_{i=0}^{R-1} (2^\alpha)^{2i(\ell_i-\ell_{i+1})} \left(\frac{2^{\ell_{R}}}{M}\right)^\epsilon.
		\end{align*}
		For the summation, we use the following estimate inductively: for $x = 2^\beta$, $\beta>0$, we have
		\begin{equation*}
			\sum_{i\leq N} x^i \leq c x^{N}.
		\end{equation*}
		for some constant $c=c(\beta)$.
		Therefore 
		\begin{align*}
			\sum_{0=\ell_0 < \ell_1\leq \dots \leq l_{R}}^{\lfloor \log M(e) \rfloor} \prod_{i=1}^{R} \pi_{2i-2}'(2^{\ell_{i-1}}, 2^{\ell_i}) \left(\frac{2^{\ell_{R}}}{M}\right)^\epsilon
			&\leq \sum_{0=\ell_0 < \ell_1\leq \dots \leq l_{R}}^{\lfloor \log M(e) \rfloor} \prod_{i=1}^{R-1} (2^\alpha)^{2\ell_i} (2^\alpha)^{-2(R-1)\ell_{R}} \left(\frac{2^{\ell_{R}}}{M}\right)^\epsilon \\
			&\leq \sum_{\ell_{R} = 1}^{\lfloor \log M(e) \rfloor} C (2^\alpha)^{2(R-1)\ell_{R}} (2^\alpha)^{-2(R-1)\ell_{R}} \left(\frac{2^{\ell_{R}}}{M}\right)^\epsilon \\
			&\leq C 2^{\epsilon\log M} M^{-\epsilon} \le C.
		\end{align*}
		
		Inserting this estimate into \eqref{eqn:adelaide}, we obtain
		\begin{equation*}
		\P\big(e\in \cup_{m=2}^\mathcal{K} \sigma_m \mid 0\leftrightarrow \partial B_n\big) \leq C\pi_3(M(e)).
		\end{equation*}

		\subsubsection{The circuits $\mathcal{C}_m$ for $m=1,\dots, \mathcal{K}$} 
		For this estimate, we again use a modified lemma from \cite{DHS16}.
		\begin{Lemma}\label{lemma:A_4}
			Suppose $C_0$ occurs and $e\in \mathcal{C}_m$ for some $1\leq m \leq \mathcal{K}$. Letting $R$ be as in Lemma \ref{lemma:A_3}, there are $0=\ell_0 < \ell_1\leq \dots \leq \ell_{R} \leq \lfloor \log M(e)\rfloor$ such that:
			\begin{enumerate}
				\item $e$ has two disjoint open arms and one closed dual arm to $\partial B(e,2^{\ell_1-1})$.
				\item For $i\geq 2$, if $2^{\ell_{i-1}} < M$, there are $2i$ disjoint open arms and one closed dual arm from $\partial B (e,2^{\ell_{i-1}})$ to $\partial B(e,2^{\ell_i-1})$.
				\item If $\ell_{R} < \lfloor \log M(e)\rfloor$, there are $2R+2$ disjoint open arms from $\partial B(e, 2^{\ell_{R}})$ to $\partial B(e, M(e))$. 
			\end{enumerate}
		\end{Lemma}
		
		\begin{proof}
			By duality, since $e \in \mathcal{C}_m$, there is a closed dual path in $\mathrm{int}(\mathcal{C}_m)$ connecting $e^*$ to a dual neighbor of $0$ if $m=1$ or to the dual of some edge $e' \in \mathcal{C}_{m-1}$ if $m>1$. Let $\ell_1$ be the minimum $\ell'$ such that $2^{\ell'} \geq \mathrm{dist}(e,e')$ if $m > 1$ (or $2^\ell \geq \mathrm{dist}(e,0)$ if $m= 1$). Then there are three arms, two open and one closed dual, from $e$ to $\partial B(e, 2^{\ell_1-1})$.
			
			If $2^{\ell_1} < M(e)$, since $e' \in B(e, 2^{\ell_1})$, we can find four open arms from $\partial B(e, 2^{\ell_1})$ to $\partial B(e, M(e))$ by following the circuits $\mathcal{C}_m$ and $\mathcal{C}_{m-1}$ from $e$ and $e'$ in both directions. We also have a closed dual arm from $e'$ to a dual neighbor of some edge $e''\in \mathcal{C}_{m-2}$ if $m>2$ or to the dual neighbor of $0$ if $m=2$. If $m>2$, we define $\ell_2$ to be the least $\ell_2\ge \ell_1$ such that $B(e, 2^{\ell_2})$ intersects $\mathcal{C}_{m-2}$. Then we have four disjoint open arms and one closed dual arm from $\partial B(e, 2^{\ell_1})$ to $\partial B(e, 2^{\ell_2-1})$. If $m=2$, we define $\ell_2 = \lfloor \log M(e)\rfloor$. There are four open arms from $\partial B(e, 2^{\ell_1})$ to $\partial B(e, M(e))$.
			
			Inductively, if $m> R$ and $\ell_{i-1} < \lfloor \log M(e) \rfloor$, we let $\ell_i$ be the least $\ell_i\ge \ell_{i-1}$ such that $B(e, 2^{\ell_i})$ intersects $\mathcal{C}_{m-i}$. Otherwise, if $\ell_{i-1} = \lfloor \log M(e)\rfloor$, we let $\ell_i = \cdots = \ell_{R} = \lfloor \log M(e)\rfloor$. If all $R$ scales $\ell_1, \dots, \ell_{R}$ are less than $\lfloor \log M(e)\rfloor$, at scale $\ell_{R}$, $B(e, 2^{\ell_{R}})$ crosses $R+1$ circuits. So there are $2R+2$ disjoint open arms from $\partial B(e, 2^{\ell_{R}})$ to $\partial B(e, M(e))$. If $m\leq  R$, then for all $i\geq m$, we define $\ell_i = \lfloor \log M(e)\rfloor$.
		\end{proof}
		
		On the event  $\{e\in \cup_{m=1}^\mathcal{K} \mathcal{C}_m\}$, by the same arguments used to bound \eqref{eqn:adelaide}, we obtain
		\begin{align}\label{eqn: rosa}
		\P\big(e\in \cup_{m=1}^\mathcal{K} \mathcal{C}_m \mid 0\leftrightarrow \partial B_n \big) 
		&\leq C\sum_{0=\ell_0 < \ell_1\leq \dots \leq \ell_{R}}^{\lfloor \log M(e) \rfloor} \prod_{i=1}^R \pi_{2i+1}(2^{\ell_{i-1}}, 2^{\ell_i-1}) \pi_{2R+2}'(2^{\ell_{R}},M) \nonumber\\
		&\leq C\pi_3(M(e)).
		\end{align}
		
		\subsubsection{Summation on $C_0$}
		As in the case of $C_0^c$, we bound the length of the path $\gamma$ defined at the beginning of this subsection.
		\begin{align*}
	    \E[\#\gamma , C_0 \mid 0\leftrightarrow \partial B_n] 
		&\leq \sum_{e\in B_n} \P\big(e\in\{\sigma_1, \sigma_{\mathcal{K}+1}\} \mid 0\leftrightarrow \partial B_n \big) \\
		&+ \sum_{e\in B_n} \P\big(e\in \cup_{m=2}^\mathcal{K} \sigma_m \mid 0\leftrightarrow \partial B_n\big) \\
		&+ \sum_{e\in B_n} \P\big(e\in \cup_{m=1}^\mathcal{K} \mathcal{C}_m \mid 0\leftrightarrow \partial B_n\big).
		\end{align*}	
		Using \eqref{eqn: rafaela}, \eqref{eqn: gabriela}, \eqref{eqn: rosa}, and summing over the values of $e$ and $M$ as in \eqref{eqn: sum-over-M}, we find
		\begin{equation*}
		\E[\#\gamma , C_0 \mid 0\leftrightarrow B_n] \leq Cn^2\pi_3(n).
		\end{equation*}
             
    \section{Proof of Theorem \ref{thm:main}} \label{section:proofofthm}
In this section, we combine the result in \cite{DHS21}, stated as Proposition \ref{prop:main}, with the estimate of Lemma \ref{lemma:DHS16} to prove Theorem \ref{thm:main}.

Let $j$ be a sufficiently large integer, and $0<\epsilon<1$ be a small parameter. In \cite{DHS21}, the authors define a sequence of events $E_j(e,\epsilon,\nu)$ such that
	\begin{itemize}
		\item each event depends only on the status of edges in the annulus
        \[Ann_j=e_x+B(2^j,2^{j+\lfloor \log \frac{1}{\epsilon}\rfloor})\]
        around $e$. Recall the vertex $e_x\in \mathbb{Z}^2$ is the lower left endpoint of the edge $e$;
		
		\item If $\sigma$ is any open path consisting of edges with three arms, two open and one dual closed extending to the outside of $Ann_j$  which includes $e$, then there exists an open path $r$ inside $Ann_j$ which is edge-disjoint from $\sigma$, but whose endpoints $u$ and $v$ lie on $\sigma$;
        
        \item Denoting by $\tau$ the portion of $\sigma$ between $u$ and $v$, we have $e\in \tau$;
        
        \item The number of edges in $r$ is at most $\nu$ times the number of edges in $\tau$. We say that $r$ is a $\nu$-shortcut around $e$:
        \[\#r \le \nu \cdot \# \tau .\]
	\end{itemize}
See \cite[Section 5 and Appendix A]{DHS21} for proofs of these statements.
The main result in \cite[Proposition 5.6]{DHS21} is that, for $0<\delta<1$ and for $\epsilon>0$ sufficiently small, there exist constants $c, \hat{c}>0$ such that
\begin{equation}\label{eqn: strict-main}
	\P\Big(\bigcap_{j= \lceil\frac{\delta}{8}\log n\rceil}^{\lfloor\frac{\delta}{4}\log n\rfloor} E_j(e, \epsilon, n^{-c})^c \mid A_3(e, n^{\delta/2})\Big) \\
	\leq 2^{-\hat{c}\frac{\delta\epsilon^4 \log n}{8\log (1/\epsilon)}}.
\end{equation}
The estimate \eqref{eqn: strict-main} implies that, conditional on the existence of 3 arms to distance $n^{\delta/2}$, there is a shortcut around $e$ which saves $n^{-c}$ edges with probability at least $1-n^{-\eta}$ for some small $\eta$. See \eqref{eqn: lara}. We will apply this to find shortcuts around the path $\gamma$ constructed in Section \ref{sec: 3-arm}.

Our main result follows from \eqref{eqn: strict-main} and the following Proposition, the analogue of Proposition 8 in \cite{DHS21}, with the lowest crossing of a box replaced by our path $\gamma$.
\begin{Prop} \label{prop:main}
	Let $e\in B_n$ and let $d\le M/4 := \frac14 \min(\mathrm{dist}(e,0),\mathrm{dist}(e,\partial B_n))$. There are choices of $c, r>0$  uniform in $d$ such that for $k\le d^{9/10}$ and any event  $E$  depending only on the status of edges in $B(e, k)$:
	\begin{equation*}
		\P(E \mid 0\leftrightarrow \partial B_n, e\in \gamma) \leq c\big(d^{-r}+\P(E \mid A_3(e,d))\big).
	\end{equation*}
\end{Prop}
Proposition \ref{prop:main} is proved in Section \ref{sec:improvements}.
\begin{proof}[Proof of Theorem \ref{thm:main}]	Choose $\delta>0$ small enough so that 
	\begin{equation}\label{eqn: adjustment}
		n^{1+2\delta} \le n^2\pi_3(n),
	\end{equation}
	and define the truncated box $\hat{B}(n) = B(n-n^\delta)\setminus B(n^\delta)$. This is possible because $\pi_3(n)\ge n^{-1+s}$ for some $s>1$; see \cite[Lemma 3.1]{DHS17} for details. 
	With our choice of $\delta>0$ and $j \in \left(\frac{\delta}{8}\log n , \frac{\delta}{4}\log n \right)$, we have, if $e\in \hat{B}(n)$:
	
	\begin{equation}\label{eqn: lara}
		\P(\text{there is no $n^{-c}$-shortcut around $e$} \mid e\in \gamma)
		\leq \P\Big(\bigcap_{j= \lceil\frac{\delta}{8}\log n\rceil}^{\lfloor\frac{\delta}{4}\log n\rfloor} E_j(e, \epsilon, n^{-c})^c \mid e\in \gamma\Big). 
		\end{equation}
	Applying Proposition \ref{prop:main} and using \eqref{eqn: strict-main}, the quantity is bounded by
	\begin{equation*}
	\begin{split}
		& cn^{-c'\delta r}+c\P\Big(\bigcap_{j= \lceil\frac{\delta}{8}\log n\rceil}^{\lfloor\frac{\delta}{4}\log n\rfloor} E_j(e, \epsilon, n^{-c})^c \mid A_3(e, n^{\delta/2})\Big) \\
		\le &~ cn^{-c'\delta r}+ 2^{-\hat{c}\frac{\delta\epsilon^4 \log n}{8\log (1/\epsilon)}}
		\end{split}
	\end{equation*}
	The last quantity is bounded by $n^{-\eta}$
	for some $\eta>0$.
	
	Along the path $\gamma$, we now choose a collection of (vertex)-disjoint $n^{-c}$-shortcuts $r_l$ such that the total length of the corresponding detoured paths $\tau_l$ is maximal. We define a path $s$ from $0$ to $\partial B_n$ by taking the union of all the shortcuts $r_l$, together with all the edges of $\gamma$ around which no shortcut exists.

    Partitioning the edges in $\gamma$ given $\{0\leftrightarrow \partial B_n\}$ into the truncated part of the box, the edges with $n^{-c}$-shortcuts, and the edges without shortcut, we estimate the expected size of $s$ as follows:
	\begin{align*}
		\E[\#s \mid 0\leftrightarrow \partial B_n] 
		&\leq Cn^{1+\delta} + n^{-c} \sum_{l} \E[\# \tau_l \mid 0\leftrightarrow \partial B_n] \\
		&+ \E[\#\{e\in \gamma \cap \hat{B}(n): \text{$e$ has no $n^{-c}$-shortcut}\} \mid 0\leftrightarrow \partial B_n] \\
		&\leq Cn^{1+\delta} + n^{-c}\E[\#\gamma \cap \hat{B}(n) \mid 0\leftrightarrow \partial B_n] + n^{-\eta}\E[\# \gamma \cap \hat{B}(n)\mid 0\leftrightarrow \partial B_n]\\
		&\le Cn^{1+\delta}+n^{-\min\{c,\eta\}}\mathbb{E}[\#\gamma\mid 0\leftrightarrow \partial B_n].
	\end{align*}
	By \eqref{eqn: adjustment} and Lemma \ref{lemma:DHS16}, we now have
	\begin{equation*}
		\E[S_n \mid 0\leftrightarrow \partial B_n] \leq \E[\# s \mid 0\leftrightarrow \partial B_n] \leq Cn^{2-\min\{c,\eta,\delta\}}\pi_3(n).
	\end{equation*}
\end{proof}

\section{Improving on $\gamma$: connecting shortcuts around three-arm points} 
	\label{sec:improvements}
    In this section, we derive the main estimate in Proposition \ref{prop:main} in Section \ref{section:proofofthm}: there are constants $c,r>0$ uniform in $d$ such that
	\begin{equation}\label{eqn: Eek-estimate-by}
		\P(E(e,k) \mid 0\leftrightarrow \partial B_n, e\in \gamma) \leq c(d^{-r}+\P(E(e,k) \mid A_3(e,d)))
	\end{equation} 
	where $E(e,k)$ is any event that depends only on the status of edges in the box of size $k$ centered at $e$ such that $k\le d^{9/10}$ for $d\le M/4$.
	
	It suffices to show \eqref{eqn: Eek-estimate-by} for $d = M/4$ since for all $k^{10/9}\le d\le M/4$:
	\begin{equation} \label{eqn: main-rhs}
	    \P(E(e,k)\mid A_3(e,M/4))\le C\P(E(e,k)\mid A_3(e,d)).
	\end{equation}
	
	To see \eqref{eqn: main-rhs}, we note that $E(e,k) \cap A_3(e,M/4)$ implies $E(e,k)\cap A_3(e,d) \cap A_3(e,d,M/4)$. By independence, we have
	\begin{equation*}
	    \P(E(e,k), A_3(e,M/4))\le \P(E(e,k), A_3(e,d))\P(A_3(e,d,M/4)).
	\end{equation*}
	On the other hand, by the standard gluing constructions, 
	\begin{equation*}
	    \P(A_3(e,M/4)) \ge c'\P(A_3(e,d))\P(A_3(e,d,M/4))
	\end{equation*}
	where $c'$ holds uniformly in $d$. Then, \eqref{eqn: main-rhs} follows by the definition of conditional probability.
	
	As in Section \ref{sec: 3-arm}, we split the event $\{0\leftrightarrow \partial B_n\}$ into the event $C_0$ that there exists an open circuit around the origin in $B_n$ and its complement $C_0^c$.
	In Sections \ref{sec: C0c} and \ref{sec: C0} respectively, we defined a path $\gamma$ from the origin to $\partial B_n$ on each of these events. We estimate the conditional probability in \eqref{eqn: Eek-estimate-by} by splitting $\{0\leftrightarrow \partial B_n\} \cap \{e\in \gamma\}$ into a number of cases, depending on the location of $e$ and which part of the path $\gamma$ the edge lies on. By the decomposition
    \[\{0\leftrightarrow \partial B_n\} \cap \{e\in \gamma\}=(\{e\in \gamma\}\cap C_0)\cup (\{e\in \gamma\}\cap C_0^c),\] 
    we have
	\begin{equation}\label{eqn: split-t}
	\P(E(e,k) \mid 0\leftrightarrow \partial B_n, e\in \gamma) \leq \P(E(e,k)\mid e\in \gamma, C_0^c)+\P(E(e,k)\mid e\in \gamma, C_0). 
	\end{equation}
	Hence it suffices to derive the estimate \eqref{eqn: Eek-estimate-by} with the left side replaced by $\P(E(e,k)\mid C, e\in \gamma)$, for $C=C_0$ or $C=C_0^c$. This involves intricate but standard gluing constructions using Russo-Seymour-Welsh and generalized Fortuin-Kasteleyn-Ginibre estimates. See Section \ref{subsection:gluing} for a discussion of such constructions. In the interest of brevity, we do not spell out the full details of the applications of FKG and RSW, but only indicate the relevant connections and provide figures for the reader's guidance. 
		\subsection{Estimate on $C_0^c$} \label{subsection:C_0^c}
	
	Recall from Section \ref{sec: C0c} that $C_0^c$ is the event that there is no open circuit around the origin in $B_n$. To minimize repetition, we treat this basic case carefully, and later indicate the necessary modifications to the argument for all other cases.
	
    Recall also from Section \ref{sec: C0c} that $\mathfrak{c}$ is the first closed dual path from the origin to $\partial B_n$ in a fixed deterministic ordering of paths on the dual lattice. The path $\gamma$ from the origin to $\partial B_n$ was constructed by choosing the closest open path to $\mathfrak{c}$ on the counterclockwise side. Therefore, for each edge $e\in \gamma$, there must be a closed dual path connecting $e$ and $\mathfrak{c}$, resulting in an intersection point of the two closed dual paths. 
	
	
By \eqref{eqn: split-t}, our task is to estimate
	\begin{equation}\label{eqn: camila}
	\P(E(e,k) \mid e\in \gamma, C_0^c) = \frac{\P(E(e,k), e\in \gamma, C_0^c)}{\P(e\in \gamma, C_0^c)}.
    \end{equation}

We find an upper bound for the numerator in \eqref{eqn: camila} and, by a gluing construction, a lower bound for the denominator. We distinguish two cases, depending on the location of $e$. 
	
\subsubsection{Case A} \emph{$e$ is in $B(n/2)$. } In this case, $M=\mathrm{dist}(e,0)$.
The event in the probability in the numerator in \eqref{eqn: camila}, $E(e,k)\cap C_0^c\cap \{ e\in \gamma\}$, implies the occurrence of
\begin{itemize}
\item a two-arm event in $B(M/4)$,
\item a three-arm event and $E(e,k)$ in $B(e,M/4)$,
\item a two-arm event in the annulus $B(2M, n)$.
\end{itemize}
    By independence of disjoint regions, this gives the estimate
		\begin{align} 
			\P(E(e,k), e\in \gamma, C_0^c)
			&\leq \P(A_{2,OC}(M/4))\P(E(e,k), A_{3,OOC}(e,M/4))\P(A_{2,OC}(2M,n)). \label{eqn:num1}
		\end{align}
    The estimate for the denominator in \eqref{eqn: camila} is somewhat more delicate, because we need to construct an event with probability of order
    \[\P(A_{2,OC}(M/4))\P(A_{3,OOC}(e,M/4))\P(A_{2,OC}(2M,n))\]
    which ensures the occurrence of $\{e\in \gamma\}$. The relevant construction is illustrated in Figure \ref{fig:a}.
    
	\begin{figure}[hpt]
		\centering
		\includegraphics[width=0.40\textwidth]{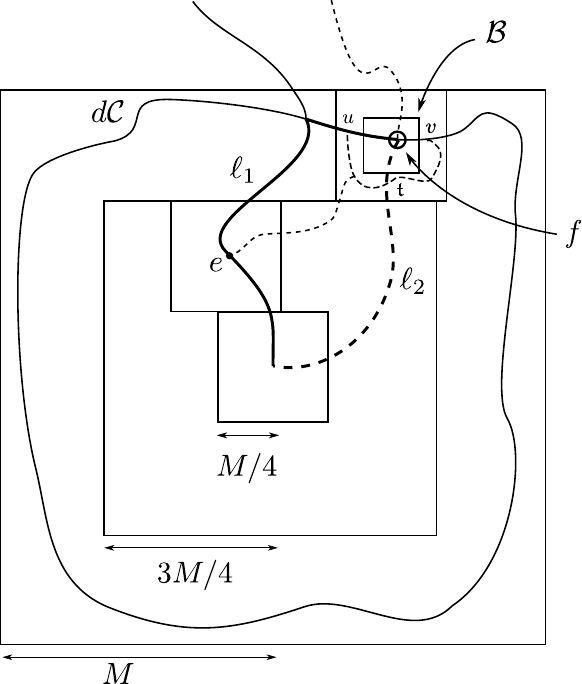}
		\caption{The construction on scale $M$ around the origin $0$ on the event $C_0^c$, Case A. ($M=\mathrm{dist}(e,0)$). Here $\ell_1$ is an open path from the origin to $\partial B_n$ and $\ell_2$ is a closed dual path from $0^*$ to $\partial B_n^*$. Such paths necessarily exist on the event $C_0^c$. The circuit in black represents $d\mathcal{C}$, a circuit with defect. The area bounded by the bold black and dotted curves is the domain $J$ in the definition of the event $D$. The construction in the figure forces the edge $e$ to belong to the open path $\gamma$. The figures are not to scale.}
		\label{fig:a}
	\end{figure}

    \begin{Definition}
        Let $D$ be the event that
        \begin{enumerate}
            \item The edge $e$ lies on an open arm $\ell_1$ from the origin to $\partial B_n$;
                                  
            \item There is an open circuit $d\mathcal{C}$ with a defect in the annulus $B(3M/4,M)$. The defect edge $f$ is contained in a box $\mathcal{B}\subset B(13M/16,15M/16)$ of side length $M/8$;

            \item There is a closed dual arc $\mathfrak{k}$ inside $d\mathcal{C}$, between the box $\mathcal{B}$ and a box $\mathcal{B}'$ with the same center as $\mathcal{B}$ and twice the side length of $\mathcal{B}$. This arc connects two edges that are dual to two open edges $u$ and $v$ on the circuit with defect $d\mathcal{C}$. The edge $u$ lies on the arc of $d\mathcal{C}$ between $f$ and the endpoint of $\ell_1$ on the circuit when the circuit is traversed in the counterclockwise direction. The edge $v$ lies between $f$ and the endpoint of $\ell_1$ when the circuit is traversed in the clockwise direction. $u$ and $v$ are connected by an arc that consists only of open edges of $d\mathcal{C}$ inside $\mathcal{B}$ as well as the (closed) defect edge $f$;
                         
            \item A dual neighbor $0^*$ of the origin is connected to $\partial B_n^*$ by a closed dual path $\ell_2$, which necessarily contains $f^*$, the dual of the defect edge in the second item. Consider the closed curve\footnote{``Closed'' here means that the curve is a continuous image of a circle.} obtained by joining the origin to $0^*$, followed by $\ell_2$, then following the circuit  with defect $d\mathcal{C}$ in the \emph{counterclockwise} direction from $f$ to the endpoint of $\ell_1$ on $d\mathcal{C}$, and finally following $\ell_1$ back to the origin. Denote the region bounded by this curve by $J$;  
            
            \item The dual edge $e^*$ is connected to the arc $\mathfrak{k}$ by a closed dual path lying inside $J$. In other words, $\mathfrak{k}$ is connected to $e^*$ to the \emph{clockwise} side of $\ell_1$. 
      
        \end{enumerate}
    \end{Definition}
    The significance of the event $D$ is in the following:
    \begin{Lemma}
        The event $D$ implies $C_0^c\cap \{e\in \gamma\}$.
    \begin{proof}
        The closed dual connections from the origin to $\partial B_n^*$ imply the occurrence of $C_0^c$. Any closed dual path from the origin to $\partial B_n$ must contain the edge $f^*$, and thus cross the dual arc $\mathfrak{k}$. This includes the arc $\mathfrak{c}$ in the definition of $\gamma$ in Section \ref{sec: 3-arm}. Thus $e^*$ is connected to $\mathfrak{c}$ by a closed dual path starting at $\mathfrak{c}$ such that the other endpoint is connected to the clockwise side of $\ell_1$. The edge $e$ must thus be part of the counterclockwise closest open path to $\mathfrak{c}$ in $B_n$, so $e\in \gamma$.
    \end{proof}
    \end{Lemma}
    
    Standard gluing constructions using generalized FKG as in Section \ref{subsection:gluing} and \cite[Section 5]{DHS17} give the following.
    
    \begin{Lemma}\label{lem: C0c-case-A-lwrbd}
        There is a constant $c>0$ such that, uniformly in the location of $e\in B_n$, we have
        \begin{equation*}
        \mathbb{P}(D)\ge c\P(A_{2,OC}(M/4))\P(A_{3,OOC}(e,M/4))\P(A_{2,OC}(2M,n)).
        \end{equation*}
    \begin{proof}
    The proof involves a repeated application of the generalized FKG \eqref{eqn: FKG} and RSW \eqref{eqn: RSW} estimates, as well as Proposition \ref{prop: smoothness}, to construct the connections indicated in Figure \ref{fig: drawing3} and force the occurrence of the event $D$, following two general principles:
    \begin{itemize}
    \item connections across boxes or annuli with aspect ratio on a fixed scale (either $n$ or $M$) have probabilities lower bounded by constants independent of $n$, $M$, and
    \item open (resp. dual closed) connections between different scales $n_1\ll n_2$ have probability costs comparable to arm events across the annulus $B(n_1,n_2)$.
    \end{itemize}
    
    For the construction inside the box $\mathcal{B}$, which contains the defect edge $f$, we use the second moment method. See Figure \ref{fig: drawing14}. Denoting by $\mathcal{N}$ the number of closed dual edges $e'$ in the box of side length $M/16$ with the same center as $\mathcal{B}$, such that $e'$ is connected to intervals $I_1$ and $I_2$ on the top, resp. bottom side of $\mathcal{B}'$ by two disjoint closed dual arms, and connected to $I_3$ and $I_4$ on the left, resp. right side of $\mathcal{B}'$ by two disjoint open arms, we have
    \[cM^2\pi_{4,OCOC}(M) \le \mathbb{E}[\mathcal{N}]\le CM^2\pi_{4,OCOC}(M).\]
    A calculation analogous to that in \cite[Proposition 5.9]{DHS17} shows that
    \[\mathbb{E}[\mathcal{N}^2]\le C(\mathbb{E}[\mathcal{N}])^2.\]
    By the second moment method, it follows that $\mathcal{N}>0$ with positive probability, in which case there is at least one such edge $e'$. Using generalized FKG, the closed dual arms emanating from $e'$ are extended into a closed dual arm from the origin to $\partial B_n$, while the open arms are extended into the circuit with defect $d\mathcal{C}$, as shown in Figure \ref{fig:a}.
    \end{proof}                                 
    \end{Lemma}

    \begin{figure} [hpt]
		\centering
		\includegraphics[width=0.750\textwidth]{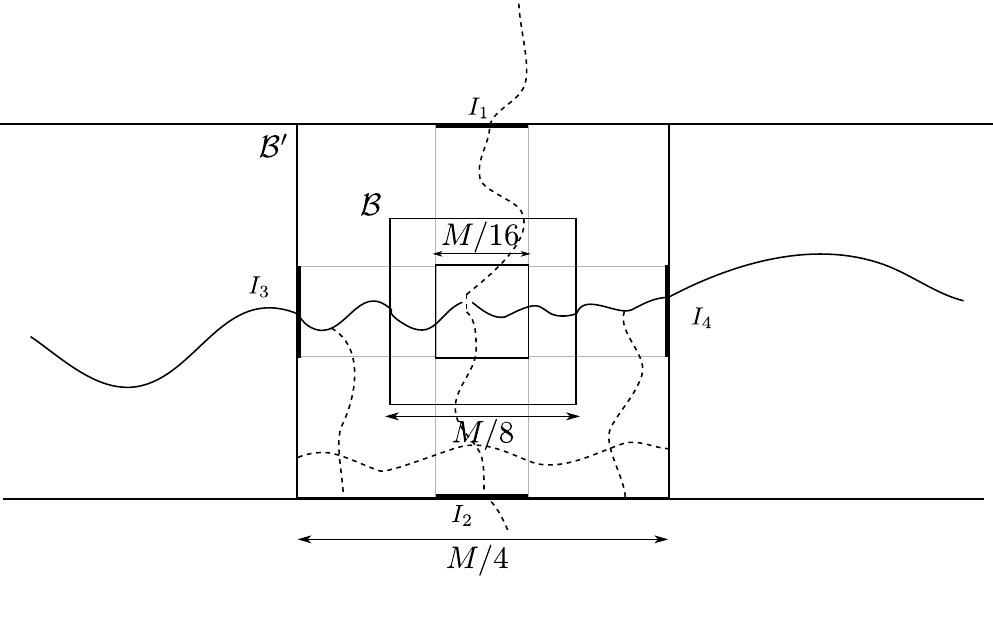}
		\caption{The construction inside the box $\mathcal{B}'$ to obtain a defect edge: we use the second moment method to show that, with probability bounded below uniformly in $n$, there is an edge inside a box of side length $M/16$ with four alternating arms (open, dual closed, open, dual closed) connected to prescribed sides of the box $\mathcal{B}'$. The open connections are then extended into an open circuit with defect. The closed dual connections are extended into an arm from the origin to $\partial B_n$. The overall cost of all connections indicated in this figure is bounded below by a constant.}
		\label{fig: drawing14}
	\end{figure}
    Combining the previous two lemmas and using \eqref{eqn: smoothness}, we obtain that the denominator in \eqref{eqn: camila} is bounded below:
    \[\P(e\in \gamma,C_0^c)\ge c\P(A_{2,OC}(M/4))\P(A_{3,OOC}(e,M/4))\P(A_{2,OC}(2M,n)), \]
    where $c>0$ is a positive constant.
    Together with the upper bound \eqref{eqn:num1}, the above implies the estimate
    \[\P(E(e,k) \mid C_0^c, e\in \gamma)\le C\P(E(e,k)\mid A_{3,OOC}(e,M/4)),\]
    in Case A, $M=\mathrm{dist}(e,0)$.

\subsubsection{Case B: $e$ is closer to $\partial B_n$.} In this case, $M = \mathrm{dist}(e,\partial B_n)$. 

To estimate the numerator in \eqref{eqn: camila}, we separate $B_n$ into three regions: the inner box $B(n/4)$, the small box $B(e,M/4)$ around $e$, and the region $B(n/4,n)\setminus B(e,M/4)$.

	\begin{figure} [hpt]
		\centering
		\includegraphics[width=0.50\textwidth]{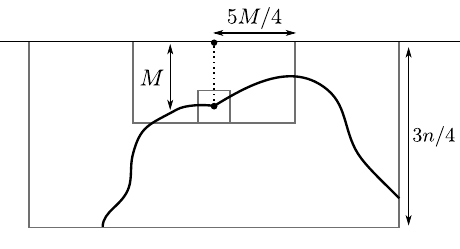}
		\caption{The existence of two open arms to distance $n$ from an edge $e$ at distance $M$ from $\partial B_n$ implies the two-arm event $A_{2,OO}(M/4)$ and the half-plane event $A_{2,OO}^{hp}(p(e),5M/4,3n/4)$.}
		\label{fig: half-plane}
	\end{figure}
	
Denote by $p_x(e)$ the projection onto $\partial B_n$ of the lower left endpoint of $e$ (the first of the two endpoints in the lexicographic order on $\mathbb{Z}^2$) along the $x$-axis, and by $p_y(e)$ the projection of the lower left endpoint onto $\partial B_n$ along the $y$-axis. Finally, let $p(e)$ be the $\ell^2$ projection of the lower left endpoint onto $\partial B_n$.

We make extensive use of the \emph{half-plane} events $A_{k,\sigma}^{hp}$. Their relevance here is illustrated in Figure \ref{fig: half-plane}. For example, if $e\in B_n$ is at distance $M\ll n$ from the boundary and has two open arms to distance of order $n$, then this implies the simultaneous occurrence of the two-arm event $A_{2,OO}(e,M/4)$ \emph{and} the occurrence inside $B_n\cap B(p(e),5M/4,3n/4)$ of two open arms. This last event is the half-plane arm event $A_{2,OO}^{hp}(e,5M/4,3n/4)$.

To simplify the proofs of the estimates below, we will assume that
\[M=\mathrm{dist}(e,\partial B_n)=\mathrm{dist}(e,p(e))\ll \max\{\mathrm{dist}(e,p_x(e)),\mathrm{dist}(e,p_y(e))\}\ge cn.\]
This amounts to assuming that the edge $e$ is not close to a corner of the box $\partial B_n$. The case when $e$ is close to a corner is dealt with by similar constructions to those in this section, involving the use of quarter-plane arm events instead of half-plane events.

The event $\{E(e,k), e\in \gamma, C_0^c\}$ implies
\begin{itemize}
\item the existence of two arms, one open and one closed dual arm in $B(n/4)$; the open arm connects the origin to $\partial B(n/4)$, the closed dual arm connects a dual neighbor of the origin to $\partial B(n/4)^*$,

\item the joint occurrence, inside $B(e,M/4)$, of $E(e,k)$ and a three-arm event (two open arms, one closed dual) from $e$ to distance $M/4$, and

\item the existence, inside the region $B_n\cap (B(p(e),3n/4))\setminus B(e,5M/4)$ of two arms, one open and one closed dual, from $\partial B(e,5M/4)$ to $\partial B(p(e),3n/4)$. Here $p(e)$ is the orthogonal projection of the lower left endpoint of the edge $e$ onto the boundary $\partial B_n$.
\end{itemize}
By independence, this gives the upper bound
	\begin{align} \label{eqn: num2}
	\P(E(e,k), e\in \gamma, C_0^c)
	&\leq \P(E(e,k),A_{3,OOC}(e,M/4), A_{2,OC}^{hp}(p(e), 5M/4, 3n/4), A_{2,OC}(n/4)) \nonumber \\
	&= \P(E(e,k), A_{3,OOC}(e,M/4))\P(A_{2,OC}^{hp}(p(e),5M/4, 3n/4))\P(A_{2,OC}(n/4)).
	\end{align}

For the lower bound, we introduce an event $F$ that implies the event $\{e\in \gamma\} \cap C_0^c$ when $M = \mathrm{dist}(e,\partial B_n)$. See Figure \ref{fig: drawing3} for an illustration.
                
\begin{Definition}
    The event $F$ is defined by the simultaneous occurrence of the following:
    \begin{enumerate}
        \item the edge $e$ lies on an open arm $\ell_1$ from the origin to $\partial B_n$;
                                  
        \item there is an open circuit $d\mathcal{C}$ with a defect in the annulus $B(n/8,n/4)$. The defect edge $f$ is contained in a box $\mathcal{B}\subset B(5n/32,7n/32)$ of side length $n/16$;

        \item there is a closed dual arc $\mathfrak{k}$ outside $d\mathcal{C}$ in the annulus between the box $\mathcal{B}$ and a box $\mathcal{B}'$ with the same center as $\mathcal{B}$ and twice the side length of $\mathcal{B}$. This arc connects two edges that are dual to two open edges $u$ and $v$ on the circuit with defect $d\mathcal{C}$. The edge $u$ lies on the arc of $d\mathcal{C}$ between $f$ and the endpoint of $\ell_1$ on the circuit when the circuit is traversed in the counterclockwise direction. The edge $v$ lies between $f$ and the endpoint of $\ell_1$ when the circuit is traversed in the clockwise direction. $u$ and $v$ are connected by an arc that consists only of open edges of $d\mathcal{C}$ inside $\mathcal{B}$ as well as the (closed) defect edge $f$;
                         
        \item a dual neighbor $0^*$ of the origin is connected to $\partial B_n^*$ by a closed dual path $\ell_2$, which necessarily contains the dual edge $f^*$, the dual to the defect edge in the first item;
        Consider the closed curve formed by concatenating $f^*$ and the portion of $\ell_2$ from $f^*$ to $\partial B_n$, then following $\partial B_n$ in the \emph{counterclockwise} direction from the endpoint of $\ell_2$ on $\partial B_n$ to the endpoint of $\ell_1$ on $d\mathcal{C}$, and finally following $d\mathcal{C}$ in the clockwise direction until one reaches $f$. Denote by $J$ the region bounded by this curve;

        \item The dual edge $e^*$ is connected to the arc $\mathfrak{k}$ by a closed dual path lying inside $J$. In other words, $\mathfrak{k}$ is connected at $e^*$ to the \emph{clockwise} side of $\ell_1$. 
    \end{enumerate}
\end{Definition}
By the same argument as for Lemma \ref{lem: C0c-case-A-lwrbd}, we obtain the following.
\begin{Lemma}\label{lem: C0c-case-B-lwrbd}
Let $e\in B_n$ and suppose that $M=\mathrm{dist}(e,\partial B_n)$. The event $F$ implies $\{ e\in \gamma \}\cap C_0^c$.
\end{Lemma}

\begin{figure} [hpt]
	\centering
	\includegraphics[width=0.55\textwidth]{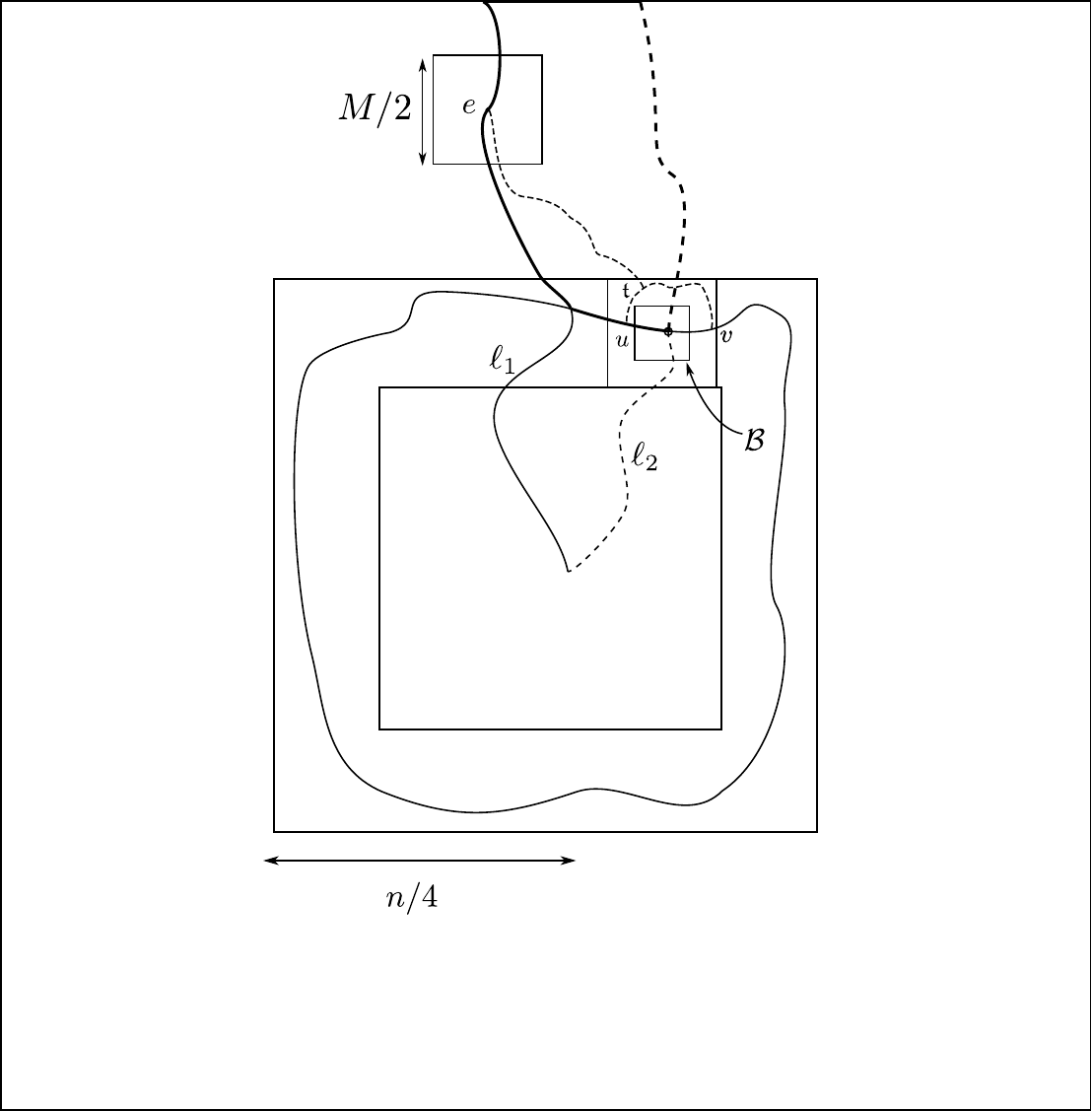}
	\caption{The construction in $B_n$ on the event $C_0^c$ in Case B ($M=\mathrm{dist}(e,\partial B_n)$). As in Figure \ref{fig:a}, $\ell_1$ and $\ell_2$ are open, respectively closed dual paths from the origin to distance $n$ and the circuit in black represents a closed dual circuit with defect. The construction forces $e\in\gamma$.}
    \label{fig: drawing3}
\end{figure}

By standard gluing constructions illustrated in Figure \ref{fig: drawing3}, one has the following:
\begin{Lemma}
  There  is a constant $c>0$ such that, if $M=\mathrm{dist}(e,\partial B_n)$ then, uniformly in the location of $e$, we have
\begin{equation*}
\mathbb{P}(F)\ge c\mathbb{P}(A_{3,OOC}(e,M/4))\mathbb{P}(A_{2,OC}^{hp}(p(e),5M/4,3n/4))\mathbb{P}(A_{2,OC}(n/4)).
\end{equation*}
\end{Lemma}
Taken together, the last two lemmas give a lower bound for the denominator in \eqref{eqn: camila} in case $M=\mathrm{dist}(e,\partial B_n)$:
\begin{equation}\label{eqn: denom-C0-c-caseB}
\P(e\in \gamma,C_0^c)\ge c\mathbb{P}( A_{3,OOC}(e,M/4))\mathbb{P}(A_{2,OC}^{hp}(p(e),5M/4,3n/4))\mathbb{P}(A_{2,OC}(n/4)).
\end{equation}

Combining the upper bound \eqref{eqn: num2} with the lower bound \eqref{eqn: denom-C0-c-caseB}, we find
\[\P(E(e,k)\mid e\in \gamma, C_0^c)\le C\mathbb{P}(E(e,k)\mid A_3(e,M/4)).\]

\subsection{Estimate on $C_0$} \label{subsection:C_0} When there exists at least one open circuit around the origin ($C_0$ occurs), the path $\gamma$, defined in Section \ref{sec: C0}, consists of portions of the successive innermost open circuits $\mathcal{C}_1,\ldots,\mathcal{C}_{\mathcal{K}}$ around the origin, as well as open paths $\sigma_1,\ldots, \sigma_{\mathcal{K}+1}$ joining the origin to $\mathcal{C}_1$, the successive circuits to each other, and finally $\mathcal{C}_{\mathcal{K}}$ to $\partial B_n$.

        We write
\begin{equation*}
\mathbb{P}(E(e,k)\mid e\in \gamma,C_0)=\frac{\P(E(e,k), e\in\gamma, C_0)}{\P(e\in \gamma, C_0)}.
\end{equation*}
The upper bound for the numerator will be obtained by a union bound along the decomposition
\[\gamma\subset \big(\cup_{m=1}^{\mathcal{K}}\mathcal{C}_m\big)\cup \big(\cup_{m=1}^{\mathcal{K}+1}\sigma_m\big),\]
using estimates close to those obtained in Section \ref{sec: 3-arm} for the volume $\# \gamma$. For the denominator, we use
\begin{equation}\label{eqn: rosalina}
\P(e\in \gamma, C_0) \geq \P(e\in \sigma_1, C_0)+ \P(e\in \cup_{m=2}^{\mathcal{K}}\sigma_m, C_0)+\P(e \in \sigma_{\mathcal{K}+1}, C_0).
\end{equation}
We then obtain lower bounds on the terms on the right side of \eqref{eqn: rosalina} by RSW/FKG constructions which force a given edge to belong to one of the portions of $\gamma$. As in the previous case, the constructions used depend on the location of the edge $e$ in $B_n$.

\subsubsection{Case A: the edge $e$ is closer to the origin than to $\partial B_n$}
In this case, we have
\[M=\mathrm{dist}(e,0).\]
We write
\begin{equation}\label{eqn: conditional-union-bd}
  \begin{split}
  \frac{\P(E(e,k), e\in\gamma, C_0)}{\P(e\in \gamma, C_0)}&\le \P(E(e,k)\mid e\in \sigma_1, C_0)+\P(E(e,k)\mid e\in \sigma_{\mathcal{K}+1}, C_0)\\
  &\quad +\P\big(E(e,k)\mid e\in \cup_{m=2}^{\mathcal{K}}\sigma_m,  C_0\big)+\frac{\P\big(E(e,k),e\in \cup_{m=1}^{\mathcal{K}}\mathcal{C}_m,C_0)}{\P(e\in \gamma,C_0)}.
  \end{split}
 \end{equation}

\emph{Estimate for $e\in \sigma_1$}.  We estimate the conditional probability
\begin{equation}\label{eqn: sigma-1-ratio}
\P(E(e,k)\mid e\in\sigma_1, C_0).
\end{equation}
When $M=\mathrm{dist}(e,0)$ ($e$ is closer to the origin), the event $\{E(e,k), e\in \sigma_1, C_0\}$ implies
\begin{itemize}
\item a two-arm event in $B(M/4)$: an open arm from the origin to $\partial B(M/4)$, and a closed dual arm from a dual neighbor of the origin to $\partial B(M/4)^*$;
\item the occurrence of $E(e,k)$ and a three-arm event (two open, one closed dual)  in $B(e,M/4)$ and
\item an open arm in the annulus $B(2M, n)$ from $\partial B(2M)$ to $\partial B_n$.
\end{itemize}
By independence of the regions involved, we obtain the upper bound
        \begin{equation}\label{eqn: sigma1-uppr}
          \P(e\in \sigma_1, C_0, E(e,k))\le \P(A_{2,OC}(M/4))\P(E(e,k), A_{3,OOC}(e,M/4))\P(A_{1,O}(2M,n)).
          \end{equation}
 We bound the denominator in \eqref{eqn: sigma-1-ratio}, $\P(e\in \sigma_1, C_0)$, below using a construction analogous to that in Section \ref{subsection:C_0^c}, Case A, but inside an open circuit on scale $M$. See Figure \ref{fig: brian}. This forces $e\in \sigma_1$. Thus:
\begin{equation}\label{eqn: sigma1-lwr}
\P(e\in\sigma_1, C_0)\ge c\P(A_{2,OC}(M/4))\P(A_{3,OOC}(e,M/4))\,\P(A_{1,O}(2M,n)).
\end{equation}
Combining \eqref{eqn: sigma1-uppr} and \eqref{eqn: sigma1-lwr}, we have
\[\P(E(e,k)\mid e \in \sigma_1,C_0)\le C\P(E(e,k)\mid A_3(e,M/4)).\]

\begin{figure}
    \centering
    \includegraphics[width=0.5\textwidth]{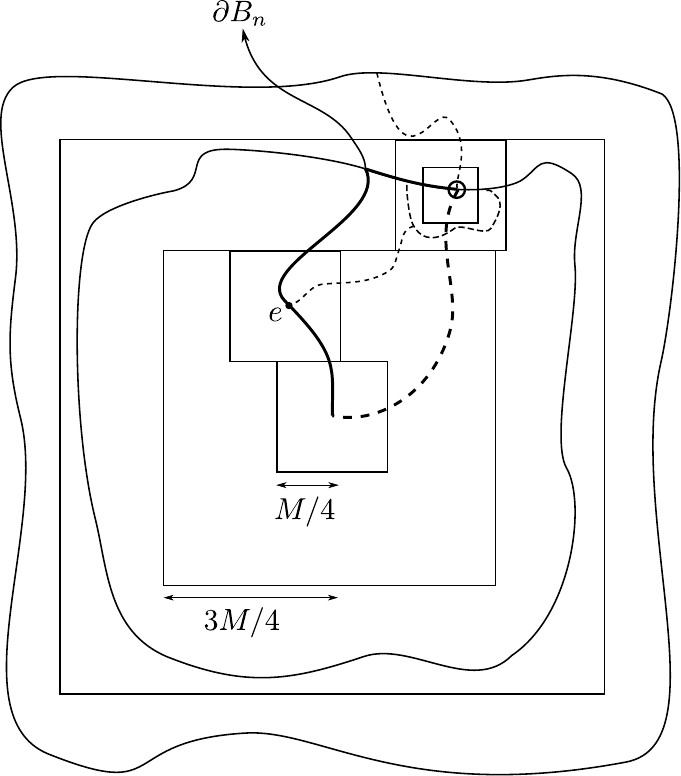}
    \caption{The construction leading to the lower bound \eqref{eqn: sigma1-lwr}.}
    \label{fig: brian}
\end{figure}

\emph{Estimate for $\sigma_{\mathcal{K}+1}$.} We now estimate the conditional probability
    \[\P(E(e,k)\mid e\in \sigma_{\mathcal{K}+1}, C_0).\]
	

	
When $e$ is closer to the origin, the event $\{E(e,k), e\in \sigma_{\mathcal{K}+1}, C_0\}$ implies
    \begin{itemize}
    \item the joint occurrence, inside $B(e,M/4)$, of the event $E(e,k)$, and a three-arm event $A_{3,OOC}(e,M/4)$;
    \item a two-arm event (one open arm and one closed dual arm) in the annular region $B(2M, n)$; and
    \item the existence of an open arm connection in $B(M/4)$.
    \end{itemize}
    This gives the upper bound:
        \begin{equation*}
          \P(E(e,k),e\in \sigma_{\mathcal{K}+1},C_0)\le \P(A_{1,O}(M/4))\P(E(e,k),A_{3,OOC}(e,M/4))\P(A_{2,OC}(2M,n)).  \end{equation*}
	By the construction illustrated in Figure \ref{fig: drawing-sun}, we also have the lower bound
    \begin{equation}\label{eqn: new-lower}
        \P(e \in \sigma_{\mathcal{K}+1},C_0)\ge c\P(A_{1,O}(M/4))\P(A_{3,OOC}(e,M/4))\P(A_{2,OC}(2M,n)).
    \end{equation}
    Combining the previous two estimates, we obtain
    \[\P(E(e,k)\mid e\in \sigma_{\mathcal{K}+1}, C_0)\le C\P(E(e,k)\mid A_3(e,M/4)).\]
    \begin{figure} [hpt]
		\centering
		\includegraphics[width=0.5\textwidth]{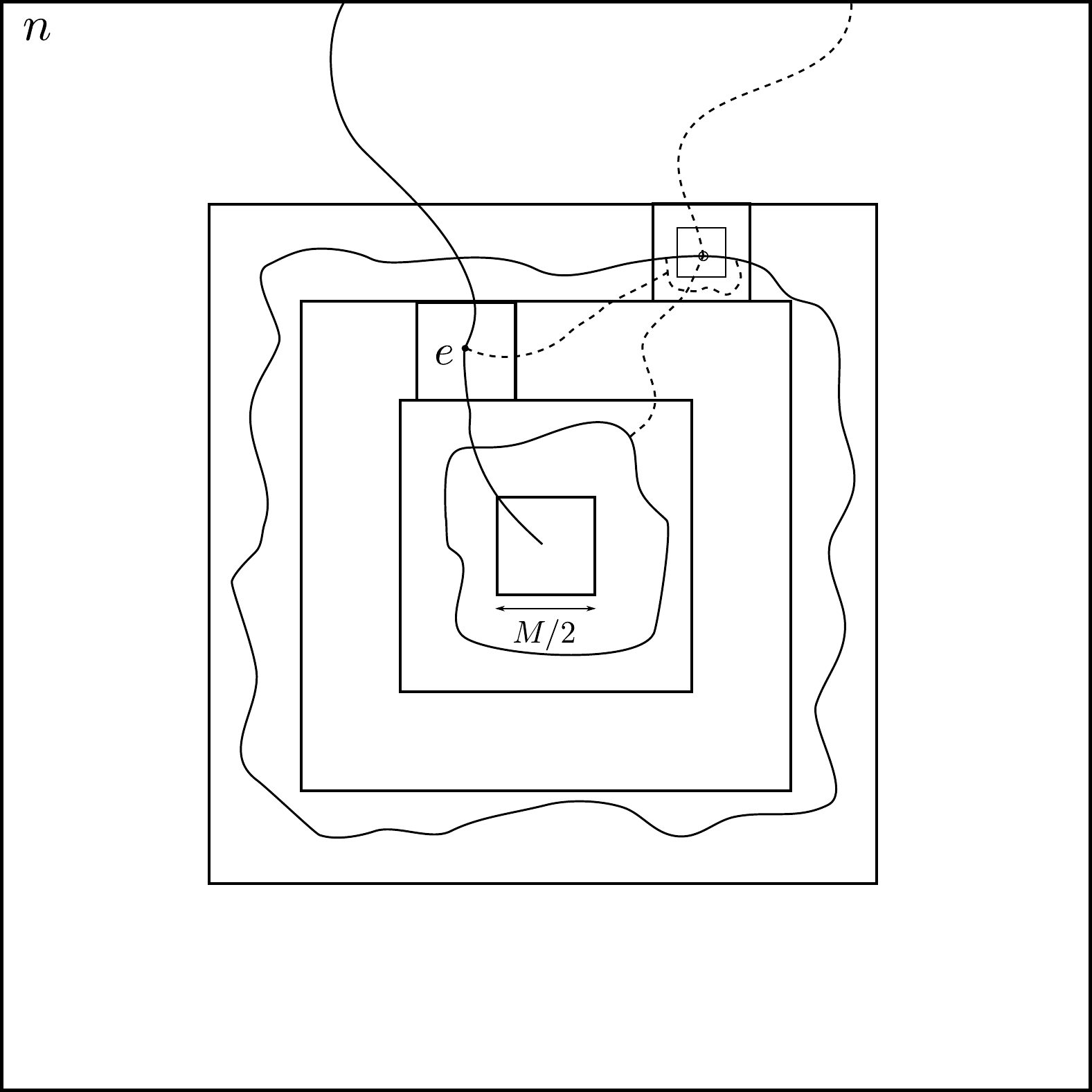}	
		\caption{The construction for the lower bound \eqref{eqn: new-lower}.}
        \label{fig: drawing-sun}
    \end{figure}	
                      
\emph{The intermediate paths $\sigma_m$ for $m=2,\dots,\mathcal{K}$.} \label{subsubsection:middlelayers}
We estimate the conditional probability
\begin{equation}\label{eqn: intermediate}
\P( E(e,k)\mid e\in\cup_{m=2}^{\mathcal{K}}\sigma_m, C_0).
\end{equation}
Here we use a modification of Lemma \ref{lemma:A_3} from Section \ref{sec: 3-arm} to obtain an upper bound. Let $k'=k^{12/11}$, so that $k'\le (M/4)^{54/55}\ll M$.
  
\begin{Lemma}\label{lem: catarina} Suppose $e\in \sigma_m$ for some $m=2,\ldots, \mathcal{K}$, and both $A_{3,OOC}(e,k)$ and $E(e,k)$ occur. Then either (i) there exists $\log (k') < \ell_1 \le \cdots \le \ell_R \le \lfloor\log (M/4)\rfloor$ such that
\begin{enumerate}
    \item $E(e,k)$ and $A_{3,OOC}(e,k')$ occur;
    \item there are two open arms and one closed dual arm, inside $B(e,k',2^{\ell_1 - 1})$, from $\partial B(e,k')$ to $\partial B(e,2^{\ell_1 - 1})$;
    \item for $i\ge 2$, if $\ell_{i-1} < \lfloor \log(M/4)\rfloor$, there are $2i$ open arms and one closed dual arm from $\partial B(e,2^{\ell_{i-1}})$ to $\partial B(e,2^{\ell_i - 1})$; 
    \item if $\ell_R < \lfloor \log (M/4)\rfloor$, there are $2R+2$ disjoint open arms from $\partial B(e,2^{\ell_R})$ to $\partial B(e,M(e))$; and
    \item there is an open arm from the origin to distance $M/4$ and one open arm from $\partial B(2M)$ to $\partial B_n$.
\end{enumerate}
or (ii) there are $0 = \ell_0 \le \ell_1 \le \cdots \le \ell_{R} \le \lfloor\log (M/4)\rfloor$ and $\ell_1\le \log(k')$ such that
\begin{enumerate}
    \item $A_{3,OOC}(e, 2^{\ell_1-1})$ occurs;
    \item for $i\ge 2$, if $\ell_{i-1}<\lfloor\log (M/4)\rfloor$, there are $2i$ disjoint open arms and one closed dual arm from $\partial B(e, 2^{\ell_{i-1}})$ to $\partial B(e,2^{\ell_i-1})$;
    \item if $\ell_R < \lfloor\log (M/4)\rfloor$, there are $2R+2$ disjoint open arms from $\partial B(e,2^{\ell_{R}})$ to $\partial B(e,M/4)$; and
    \item there is an open arm from the origin to distance $M/4$ and one open arm from $\partial B(2M)$ to $\partial B_n$.
\end{enumerate}
\end{Lemma}
The second case occurs when there is no space between $\mathcal C_m$ and $\mathcal C_{m+1}$ to fit the box $B(e, k')$.

The estimate in case $(ii)$ uses an adaptation of the idea in Lemma \ref{lemma:A_3}:
\begin{align*}
    \P(E(e,k), & \, e\in \cup_{m=2}^{\mathcal{K}}\sigma_m, C_0, \text{case } (ii))\\
    \le & C \P(A_{1,O}(M/4)) \P(A_{1,O}(2M, n)) \\
    &\times \sum_{0=l_0 \le \ell_1 \le \log(k')}\pi_3(2^{\ell_1})\sum_{\ell_1\le \ell_2\le \ldots \le \ell_{R}} \prod_{i=2}^R \pi_{2i+1}(2^{\ell_{i-1}}, 2^{\ell_i-1}) \pi_{2R+2}'(2^{\ell_R},M/4).
\end{align*}
For $i=3,\dots, R$, we use Reimer's inequality in the form
\begin{equation*}
    \pi_{2i+1}(2^{\ell_{i-1}}, 2^{\ell_i-1}) \le \pi_5(2^{\ell_{i-1}}, 2^{\ell_i-1}) \pi_{2i-4}'(2^{\ell_{i-1}}, 2^{\ell_i-1}).
\end{equation*}
By  similar computations to those in Section \ref{sec: intermediate}, we obtain a bound for the inner sum over $\ell_1,\ldots,\ell_R$ of $c^{R-2}\pi_5(2^{\ell_1},2^{\ell_R})$.
Using Reimer's inequality, quasi-multiplicativity, and \eqref{eqn:stopcondition}, we are left with 
\begin{align*}
    &\P(E(e,k), e\in \cup_{m=2}^{\mathcal{K}}\sigma_m, C_0, \text{case } (ii)) \\
    \quad \le~ & C \P(A_{1,O}(M/4)) \P(A_{1,O}(2M, n)) \sum_{0\le \ell_1 \le \log(k')} \pi_3(2^{\ell_1}) \sum_{\ell_R \le \log(M/4)} \pi_5(2^{\ell_1},2^{\ell_R}) \pi_3(2^{\ell_R}, M/4) \left(\frac{2^{\ell_R}}{M/4}\right)^\epsilon \\
    \quad \le~ & C  \P(A_{1,O}(M/4)) \P(A_{1,O}(2M, n)) \pi_3(M/4) \sum_{\ell_R\le \log(M/4)} \left(\frac{2^{\ell_R}}{M/4}\right)^\epsilon \sum_{\ell_1 \leq \log(k')} 2^{-2\alpha(\ell_R-\ell_1)} \\
    \quad \le~ & C  \P(A_{1,O}(M/4)) \P(A_{1,O}(2M, n)) \pi_3(M/4) \left(\frac{4k'}{M}\right)^{2\alpha}.
\end{align*}
Since $k$ is chosen to satisfy $k\le (M/4)^{9/10}$, the right side of the final quantity above is then bounded by
\[C  \P(A_{1,O}(M/4)) \P(A_{1,O}(2M, n))\pi_3(M/4) (M/4)^{-r}\]
for some $r>0$.

For the first case $(i)$ of Lemma \ref{lem: catarina}, similarly to Lemma \ref{lemma:A_3} and in the previous case, the probability of items (2) -- (4) is bounded by
\begin{equation*}
    \sum_{\lceil\log(k')\rceil < \ell_1 \le \dots \le \ell_R}^{\lfloor \log (M/4)\rfloor} \P(A_{3, OOC}(e, k', 2^{\ell_1-1})) \prod_{i=1}^{R-1} \pi_{2i+1}(2^{\ell_i}, 2^{\ell_{i+1}-1})) \pi'_{2R+2}(2^{\ell_R}, M/4) \le C \pi_3(k', M/4).
\end{equation*}

 We now claim:
\begin{equation}\label{eqn: arm-sep}
\begin{split}
    &\P(E(e,k), A_{3, OOC}(e, k')) \P(A_{3, OOC}(k', M/4))\\
    \le~&C \P(E(e,k), A_{3, OOC}(e, M/4)) + C(k/k')^\delta \pi_3(M/4),
\end{split}
\end{equation}
for some $\delta>0$. Note that if we remove the event $E(e,k)$ from the probability on both sides, \eqref{eqn: arm-sep} follows immediately from arm separation and gluing, as summarized in Section \ref{subsection:gluing}. However, the presence of the event $E(e,k)$ could bias the probability of $A_{3,OOC}(e,k')$ and thus prevent arm separation on the boundary of $B(e,k')$. This can be circumvented by another classical technique: decoupling at interfaces. The interfaces we consider here are closed dual circuits around the origin with two open edge-defects. We will be brief in our presentation. We refer the reader to \cite[Section 6]{DS11}, \cite[Proposition 6.4]{DHS16}, and \cite[Lemma 4.4]{DHS21} for related estimates proved using this technique.

We introduce the event $S_{k,k'}$ that there is a closed dual circuit with two open defects in some annulus $B(e,2^j,2^{j+1})$, $j=\lfloor \log(2k)\rfloor, \ldots, \lfloor \log(k')\rfloor$. We begin by excluding the event $S_{k,k'}$. Using the argument in \cite[Theorem 4.1]{DHS21}, there is a $\delta>0$ such that
\[\mathbb{P}(S_{k,k'}^c\mid A_{3,OOC}(e,k'))\le C(k/k')^\delta.\]
Thus,
\begin{equation}\label{eqn: inter-2}
\mathbb{P}(E(e,k), A_{3,OOC}(e,k'))\le \mathbb{P}(E(e,k), A_{3,OOC}(e,k'), S_{k,k'}) + C(k/k')^\delta \pi_3(k').
\end{equation}
For any closed dual circuit $D$ with two open defects, the two defects partition the circuit into two arcs. For the deterministic ordering fixed at the beginning of Section \ref{sec: gamma-def}, let the arc with the smallest edge be denoted $\mathrm{Arc}_1(D)$, and the other $\mathrm{Arc}_2(D)$. We define $X_-(D,i)$ be the event that $e$ is connected to $D$ by three disjoint arms: two open arms from $e$ to the two defects, and a closed dual arm to $\mathrm{Arc}_i(D)$. Similarly, we let $X_+(D,k',i)$ be the event that $D$ is connected to the boundary of the box $B(e,k')$ by three disjoint arms: two open arms emanating from the two defects, and a closed dual arm from $\mathrm{Arc}_i(D)$. 

We then partition the event $\{E(e,k), A_{3,OOC}(e,k'), S_{k,k'}\}$ at the innermost circuit with defects $\mathcal{D}$ in the region $B(e,2k,k')$:
\begin{equation} \label{eqn: orange}
\mathbb{P}(E(e,k), A_{3,OOC}(e,k'), S_{k,k'}) =\sum_D \sum_{i=1,2}\mathbb{P}\big(E(e,k), \mathcal{D}=D, X_-(D,i)\big)\mathbb{P}(X_+(D,k',i)). 
\end{equation}
By the $3$-arm analogue of external arm separation as in \cite[Lemma 6.2]{DS11}, we have
\begin{equation}\label{eqn: india}
\mathbb{P}(X_+(D,k',i))\le C\mathbb{P}(\tilde{X}_+(D,k',i)),
\end{equation}
where in the $\tilde{X}_+(D,k',i)$ the arms from the circuit with defects to $\partial B(e,k')$ are well-separated and have given landing sites. See \cite[Definitions 7 \& 9]{Nolin08} for the relevant definitions.
Inserting \eqref{eqn: india} into \eqref{eqn: orange}, we have now showed that
\[\mathbb{P}(E(e,k),A_{3,OOC}(e,k'), S_{k,k'})\le C\mathbb{P}(E(e,k),\tilde{A}_{3,OOC}(e,k')),\]
where the arms on $\partial B(e,k')$ in the event $\tilde{A}_{3,OOC}(e,k')$ are well-separated and have their endpoints in specified intervals on the boundary. We can now use standard gluing to obtain that
\[\mathbb{P}(E(e,k),\tilde{A}_{3,+,OOC}(e,k'))\pi_3(k',M)\le C\mathbb{P}(E(e,k), A_{3,OOC}(e,M)).\]
The claim \eqref{eqn: arm-sep} now follows from this and \eqref{eqn: inter-2}, using $\pi_3(k')\pi_3(k',M)\le C\pi_3(e,M)$.


Together with items (1) and (5), we have the estimate
\begin{equation}\label{eqn: decompose}
\begin{split}
  & \P(E(e,k), e\in \cup_{m=2}^{\mathcal{K}}\sigma_m, C_0, \text{case }(i))\\
  & \quad \le \P(A_{1,O}(M/4)) \P(A_{1,O}(2M, n))\P(E(e,k), A_{3,OOC}(e,k')) C \pi_3(k', M/4) \\
  & \quad \le  C'\P(A_{1,O}(M/4))\P(A_{1,O}(2M, n))\left(\P(E(e,k),A_{3,OOC}(e,M/4)) + (k/k')^\delta\pi_{3}(M/4)\right).
\end{split}
\end{equation}

Thus, we obtain the upper bound
\begin{equation}\label{eqn: upper-bd-sigmam}
\begin{split}
    &\mathbb{P}(E(e,k), e\in \cup_{m=2}^{\mathcal{K}} \sigma_m, C_0) \\
    \le &~ C\left(\P(E(e,k), A_{3,OOC}(M/4))+ M^{-r}\pi_3(M/4)\right) \P(A_{1,O}(M/4))\P(A_{1,O}(2M, n)).
\end{split}
\end{equation}

For the denominator in \eqref{eqn: intermediate}, we have the following lower bound, a consequence of the construction illustrated in Figure \ref{fig: drawing2}:
\begin{equation}\label{eqn: lwr-bd-sigmam}
    \mathbb{P}( e\in \cup_{m=2}^{\mathcal{K}}\sigma_m, C_0)\ge c\P(A_{3,OOC}(M/4))\P(A_{1,O}(M/4))\P(A_{1,O}(2M, n)).
\end{equation}

Finally, by \eqref{eqn: upper-bd-sigmam} and \eqref{eqn: lwr-bd-sigmam}, we have
\[\frac{\P(E(e,k), e\in \cup_{m=2}^{\mathcal K}\sigma_m, C_0)}{\P(e\in \cup_{m=2}^{\mathcal K}\sigma_m, C_0)}\le C(\P(E(e,k)\mid A_3(e,M/4)) + M^{-r}).\]

\begin{Definition}\label{def: arem} For later use, we denote the event defined in Lemma \ref{lem: catarina}, case (ii) by $\mathcal{A}_R(e,M/4)$, where it is always understood that $R$ is chosen such that \eqref{eqn:stopcondition} holds. The previous computations show that
    \[\mathbb{P}(\mathcal{A}_R(e,M/4))\le c M^{-r} \pi_3(M/4),\]
for some $r>0$.    
\end{Definition}

\begin{figure}[hpt]
	\centering
	\includegraphics[width=0.55\textwidth]{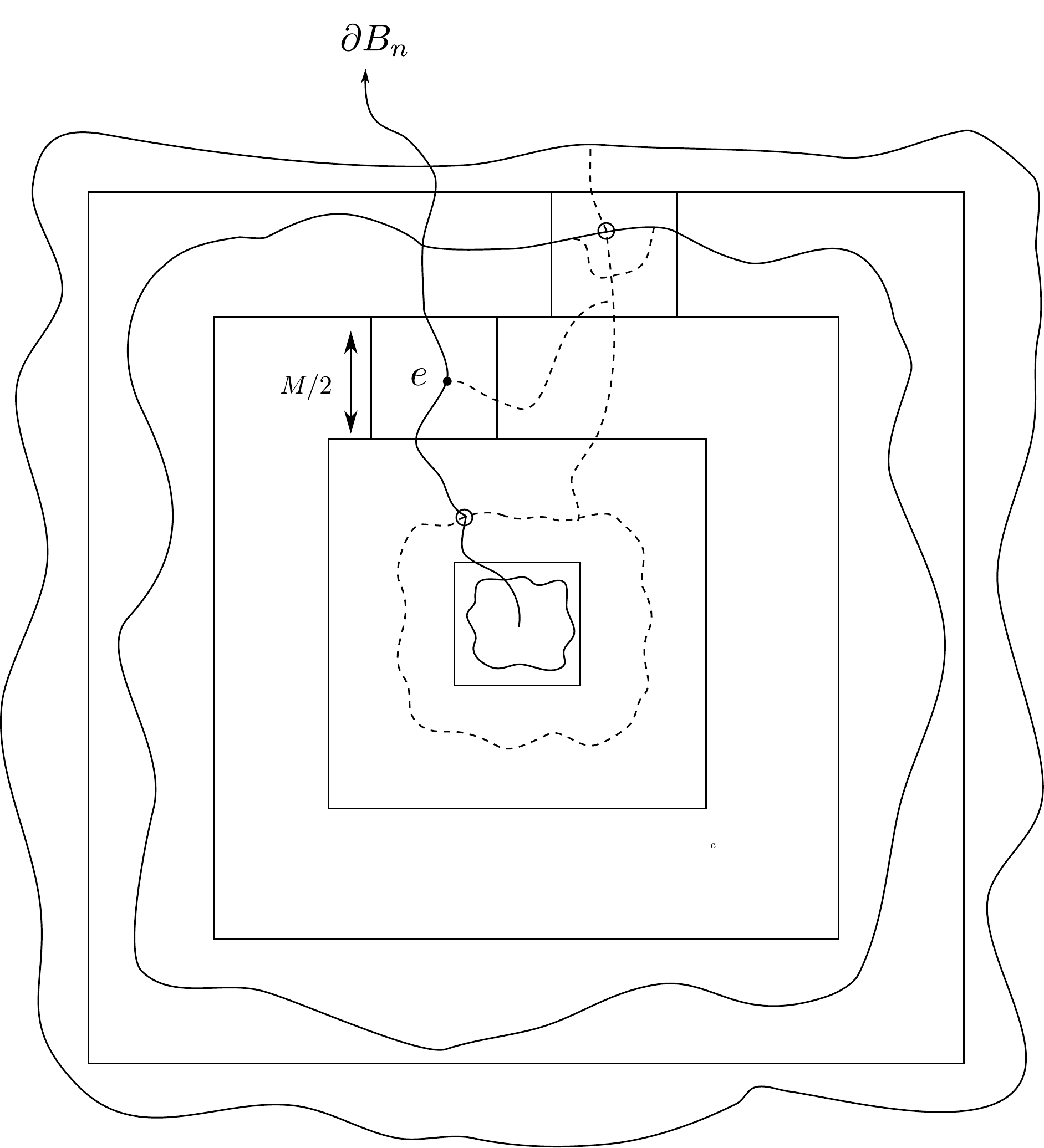}
	\caption{The construction used to prove the lower bound \eqref{eqn: lwr-bd-sigmam}.}
	\label{fig: drawing2}
\end{figure}

\emph{The circuits $\mathcal{C}_m$, $m=1,\ldots,\mathcal{K}$.} This case is dealt with similarly to the probability of the event $\{e\in \cup_{m=2}^{\mathcal{K}}\sigma_m, C_0\}$. The upper bound in this case is of the same form as \eqref{eqn: upper-bd-sigmam}:
\begin{align*}
  &\mathbb{P}(E(e,k),e\in \cup_{m=1}^{\mathcal{K}}\mathcal{C}_m, C_0)\\
  \quad &\le C(\P(E(e,k),A_{3,OOC}(e,M/4))+
 M^{-r}\pi_3(e,M/4))\P(A_{1,O}(M/4))\P(A_{1,O}(2M,n)).
\end{align*}
To control the denominator in the last term of \eqref{eqn: conditional-union-bd}, we use the lower bound \eqref{eqn: lwr-bd-sigmam}
\begin{equation*}
\P(e\in \gamma, C_0)\ge \P(e\in \cup_{m=2}^\mathcal{K} \sigma_m,C_0)\ge c\P(A_{3,OOC}(e,M/4))\P(A_{1,O}(M/4))\P(A_{1,O}(2M,n)).
\end{equation*}

\subsubsection{Case B: $e$ is closer to $\partial B_n$ than the origin.}\label{lem: mercedes}
In this case, we estimate the ratio
\begin{equation*}
\frac{\P(E(e,k), e\in\gamma, C_0)}{\P(e\in \gamma, C_0)}
\end{equation*}
using the lower bound \eqref{eqn: rosalina} instead of \eqref{eqn: conditional-union-bd}. 

\emph{Estimate for $e\in \sigma_1$}. We estimate
\[\frac{\P(E(e,k), e\in\sigma_1, C_0)}{\P(e\in \gamma, C_0)}.\]
When the edge $e$ is closer to the boundary ($M=\mathrm{dist}(e,\partial B_n)$), the following occur:
\begin{enumerate}
    \item three arms, two open and one closed dual, inside $B(e,M/4)$ from $e$ to $\partial B(e,M/4)$;
    \item four arms, three open and one closed dual, in the semi-annular region $B_n\cap B(p(e),5M/4,3n/4)$ and
    \item two arms, one open and one closed dual, from the origin to distance $n/4$.
\end{enumerate}

  By independence, this results in the bound
  \begin{equation}\label{eqn: sigmaK-upper}
    \P(e\in \sigma_1, C_0)\le \P(A_{3,OOC}(e,M/4),E(e,k))\P(A_{2,OC}(n/4))\P(A_{2,OC}^{hp}(5M/4,3n/4)).
  \end{equation}
  For the lower bound, we use the construction illustrated in Figure \ref{fig: drawing5} to obtain the estimate
    \begin{equation}\label{eqn: sigmaK-lwr}
    \begin{split}
        \P(e\in \gamma,  C_0)&\ge \P(e\in \sigma_{\mathcal{K}+1}) \\
        &\ge c\P(A_{3,OOC}(e,M/4))\P(A_{1,O}(n/4))\P(A_{2,OC}^{hp}(5M/4,3n/4))\\
        &\ge c\P(A_{3,OOC}(e,M/4))\P(A_{2,OC}(n/4))\P(A_{2,OC}^{hp}(5M/4,3n/4)).
    \end{split}
    \end{equation}
    Using \eqref{eqn: sigmaK-upper} and \eqref{eqn: sigmaK-lwr}, we have
    \[\frac{\P(E(e,k), e\in \sigma_1, C_0)}{\P(e\in \gamma, C_0)}\le C\P(E(e,k)\mid A_3(e,M/4)).\]

\emph{Estimate for $e\in \sigma_{\mathcal{K}+1}$.} When $e$ is closer to the boundary ($M=\mathrm{dist}(e,\partial B_n)$), we use the following lemma to obtain an upper bound for the probability $\P(E(e,k),e\in \sigma_{\mathcal{K}+1}, C_0)$.
	\begin{Lemma}\label{lem: star}
    	Suppose $e\in \sigma_{\mathcal{K}+1}$ and $M=\mathrm{dist}(e,\partial B_n)$. If $E(e,k)\cap C_0$ occurs, then there exists $\log(2k)< j\le l \le \log(3n/4)$ such that
    	\begin{enumerate}
        \item $E(e,k)$ and the three-arm event $A_{3,OOC}(e,M/4)$ occur;
        \item there are two arms in $B_n\cap B(p(e),5M/4,2^{j-1})$, one open and one closed dual, from $\partial B(p(e),5M/4)$ to $\partial B(p(e),2^{j-1})$;
        \item there are three arms, two open and one closed dual, in $B_n\cap B(p(e),2^j,2^{l-1})$ from $\partial B(p(e),2^j)$ to $\partial B(p(e),2^{l-1})$;
        \item there are four open arms inside $B_n\cap B(p(e),2^l,3n/4)$ and
        \item there is an open arm in $B(n/4)$ from the origin to $\partial B(n/4)$.
        \end{enumerate}
        \begin{proof}
        	On $C_0$, enumerate the successive innermost circuits as $\mathcal{C}_1,\dots,\mathcal{C}_{\mathcal{K}}$. If $e\in \sigma_{\mathcal{K}+1}$, then $e$ and $p(e)$ lie outside $\mathcal{C}_{\mathcal{K}}$. We let $j$ be the least integer such that $B(p(e),2^j)\cap \mathcal{C}_{\mathcal{K}}\neq \emptyset$, and $l\ge j$ be the least $l$ such that $B(p(e),2^l)\cap \mathcal{C}_{\mathcal{K}-1}\neq \emptyset$ if $\mathcal{K}\ge 2$. Otherwise, we set $l=\lceil \log(3n/4)\rceil$. It is now easy to check that the claims regarding the arm events hold.
        \end{proof}
    \end{Lemma}
            
Decomposing according to the distances $2^j, 2^l$, we estimate the ratio 
\[\frac{\P(E(e,k), e\in \sigma_{\mathcal{K}+1}, C_0)}{\P(A_{1,O}(n/4))\P(E(e,k), A_{3,OOC} (e,M/4))}\]
by the sum
\begin{equation}\label{eqn: decompose-2}
	\begin{split}
	&~ \sum_{\log(2k)<j\le l \le \log(n/4)}\P(A_{2,OC}^{hp}(p(e),5M/4,2^{j-1}))\P(A_{3,OOC}^{hp}(p(e),2^j,2^{l-1}))\P(A_{4,OOOO}^{hp}(2^l,3n/4))\\
	\le&~ C\sum_{\log(2k)<j\le l \le \log(n/4)}\P(A_{2,OC}^{hp}(p(e),5M/4,2^j))\P(A_{2,OC}^{hp}(p(e),2^j,2^l))\frac{2^{\epsilon j}}{2^{\epsilon l}}\P(A_{2,OO}^{hp}(2^l,3n/4))\frac{2^{2 \epsilon l}}{n^{2\epsilon}}\\
	\le&~ C'\P(A_{2,OC}^{hp}(p(e),5M/4,3n/4)).
	\end{split}
\end{equation}
This gives the upper bound
\begin{equation}\label{eqn: upper-bound-sK+1-caseA}
	\P(E(e,k), e\in \sigma_{\mathcal{K}+1},C_0)\le C'\P(A_{1,O}(n/4))\P(E(e,k),A_{3,OOC}(e,M/4))\P(A_{2,OC}^{hp}(p(e),5M/4,3n/4)).
\end{equation}
The main estimate for the denominator $\P(e\in \gamma, C_0)$ in this case is
\begin{equation}\label{eqn: lwr-bound-sK+1-caseB}
\P(e\in \gamma,C_0)\ge \P(e\in \sigma_{\mathcal{K}+1},C_0)  \ge c\P(A_{1,O}(n/4))\P(A_{3,OOC}(e,M/4))\P(A_{2,OC}^{hp}(p(e,5M/4,3n/4))).
\end{equation}
This is obtained by the construction illustrated in Figure \ref{fig: drawing5}.
\begin{figure}[hpt]
	\centering
	\includegraphics[width=0.50\textwidth]{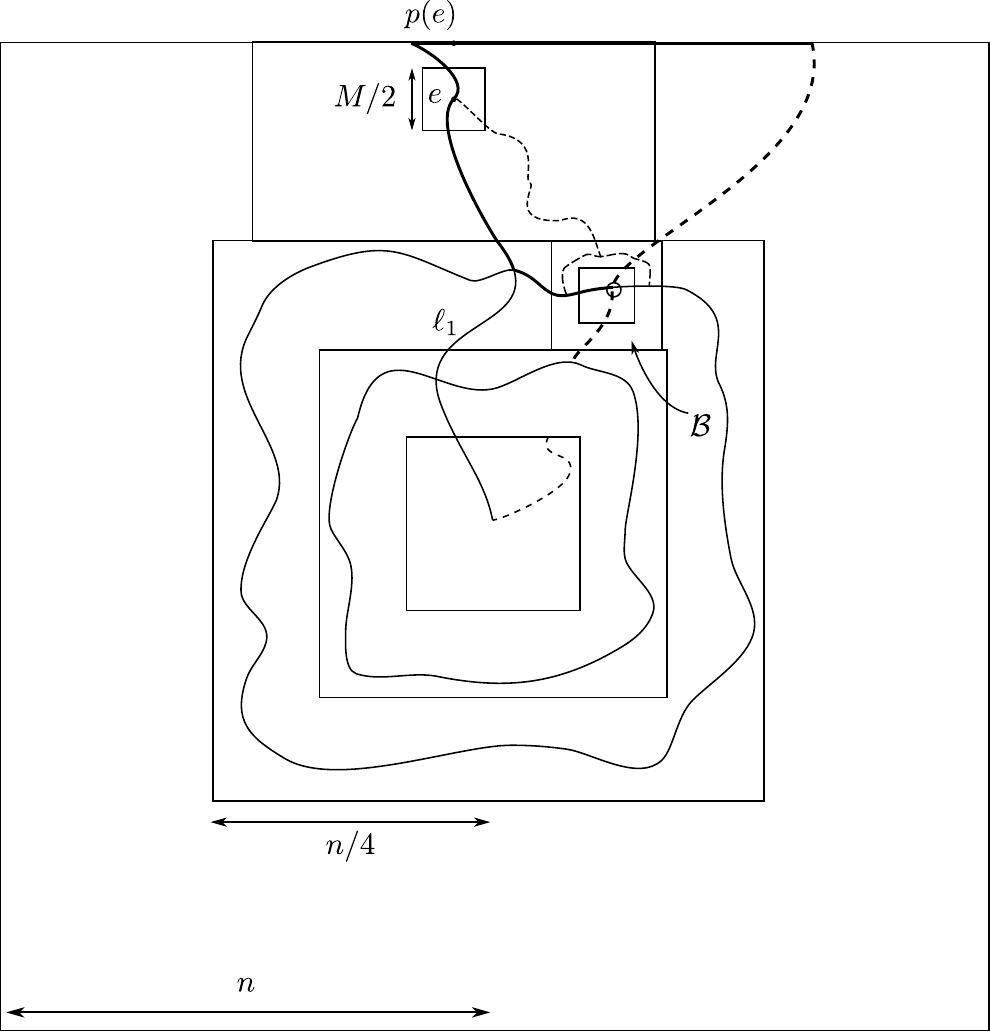}
	\caption{The construction near the boundary used to obtain the lower bound \eqref{eqn: lwr-bound-sK+1-caseB}.}
	\label{fig: drawing5}
\end{figure}
Combining \eqref{eqn: upper-bound-sK+1-caseA} and \eqref{eqn: lwr-bound-sK+1-caseB}, we obtain the desired estimate:
\[\P(E(e,k)\mid e\in \sigma_{\mathcal{K}+1},C_0)\le C\P(E(e,k) \mid A_3(e,M/4)), \quad \text{ if } M=\mathrm{dist}(e,\partial B_n).\]
	
\emph{Estimate for $e\in \cup_{m=2}^{\mathcal{K}}\sigma_m$.} When $M=\mathrm{dist}(e,\partial B_n)$, we have the upper bound
\begin{equation}\label{eqn: sigmas}
\begin{split}
    \P(E(e,k),e\in \cup_{m=2}^{\mathcal{K}}\sigma_m,C_0 )
    &\le (C\P(A_{3,OOC}(e,M/4),E(e,k)) +  M^{-r}\pi_3(e,M/4))) \\
    &\qquad \cdot \P(A_{1,O}(e,n/4))\P(A_{2,OC}^{hp}(p(e),5M/4,3n/4)).
\end{split}
\end{equation}
This estimate follows from the following lemma, which is proved along the same lines as Lemma \ref{lemma:A_3} in Section \ref{sec: 3-arm} and Lemma \ref{lem: star} in Section \ref{lem: mercedes}.
\begin{Lemma}\label{lem: either-or}
Suppose $M=\mathrm{dist}(e,\partial B_n)$. Then,
\begin{enumerate}
    \item 
    \begin{enumerate}
        \item either the event $E(e,k)$ and the three-arm event $A_{3,OOC}(e,M/4)$ occur;
        \item or the event $\mathcal{A}_R(e,M/4)$ in Definition \eqref{def: arem} occurs;
    \end{enumerate}
    \item there are two open arms in $B_n \cap B(p(e),5M/4,3n/4)$ and;
    \item there is an open arm in $B(e,n/4)$ from the origin to $\partial B(e,n/4)$.
\end{enumerate}

\end{Lemma}

\begin{proof} [Proof sketch.]
    The proof for item $(1)$ is very similar to the proof for Lemma \ref{lemma:A_3}. If there is no closed dual arm to distance $M/4$, then $\mathcal{C}_m$ or $\mathcal{C}_{m+1}$ intersects $B(e,M/4)$. We can then find a sequence of scales $\ell_1 \le \ell_2 \le \cdots \le \ell_R$ such that the annulus on each scale contains two more open arms than the preceding one. For item $(2)$, there are at least two arms consisting of portions of $\mathcal{C}_{m+1}$ in $B_n \cap B(p(e),5M/4,3n/4)$.
\end{proof}

Applying the lemma and decomposing according to the distance $j$ as in \eqref{eqn: decompose}, \eqref{eqn: decompose-2}, we obtain the estimate \eqref{eqn: sigmas}.

For the denominator, we use the lower bound \eqref{eqn: lwr-bound-sK+1-caseB} established for the event $\{e\in \sigma^{\mathcal{K}+1}\}$.
\begin{equation}\label{eqn: sigmas-lwr}
\mathbb{P}( e\in \sigma_{\mathcal{K}+1}, C_0)\ge c\P(A_{3,OOC}(e,M/4))\P(A_{2,OC}^{hp}(p(e),5M/4,3n/4))\P(A_{1,O}(n/4)).
\end{equation}
Combining \eqref{eqn: sigmas} and \eqref{eqn: sigmas-lwr}, we find
\[\frac{\P(E(e,k), e\in \cup_{m=2}^{\mathcal{K}}\sigma_m,C_0)}{\P(e\in \gamma, C_0)}\le C(\P(E(e,k)\mid A_{3,OOC}(e,M/4))+M^{-r}).\]

\emph{Estimate for $e\in \cup_{m=1}^{\mathcal{K}+1}\mathcal{C}_m$.} This case can be handled very similarly to that in the previous section. By a variant of Lemma \ref{lem: either-or}, we obtain the upper bound:
\begin{equation}\label{eqn: circuits-upper}
\begin{split}
    \mathbb{P}( E(e,k), e\in \cup_{m=1}^{\mathcal{K}+1}\mathcal{C}_m,C_0)
    &\le C(\P(E(e,k),A_3(e,M/4))\\
    &+ M^{-r}\pi_3(e,M/4))\P(A_{2,OC}^{hp}(p(e),5M/4,3n/4))\P(A_{1,O}(n/4)).
\end{split}
\end{equation}

We use the lower bound established for $\P(e\in \gamma, C_0)$ in \eqref{eqn: lwr-bound-sK+1-caseB}. Combined with the upper bound \eqref{eqn: circuits-upper}, we obtain
\[\frac{\P(E(e,k),  e\in \cup_{m=1}^{\mathcal{K}+1}\mathcal{C}_m, C_0)}{\P(e\in \gamma, C_0)}\le C\P(E(e,k)\mid A_{3,OOC}(M/4)).\]

\bibliographystyle{plain}
\bibliography{reference}
        
\end{document}